\documentclass[mnsc,nonblindrev]{oo} 

\OneAndAHalfSpacedXII 

\usepackage{amsmath}
\usepackage{url}
\usepackage{enumitem}
\usepackage[english]{babel}
\usepackage{bbm}
\usepackage{bm}
\usepackage{comment}
\usepackage{hyperref} 

\newcommand{\R}{\mathbb{R}}

\newcommand{\bsubeq}{\begin{subequations}}
\newcommand{\esubeq}{\end{subequations}}
\newcommand{\BI}{\begin{itemize}}
\newcommand{\EI}{\end{itemize}}
\newcommand{\I}{\item}

\newcommand{\cgreen}{\color{black}}
\newcommand{\cred}{\color{red!85!black}}

\newcommand{\mb}{\mathbb}

\usepackage{graphicx}
\usepackage{enumerate}
\usepackage{subfigure}
\usepackage[normalem]{ulem}

\renewcommand{\S}{\mathcal{S}} 

\newcommand{\vpi}{\boldsymbol{\pi}}

\usepackage{longtable}
\newcommand{\sm}{\setminus}

\usepackage{mathtools}

\usepackage{multirow}
\usepackage{graphicx}
\usepackage{caption}
\usepackage{breqn}
\allowdisplaybreaks

\renewcommand{\mc}{\mathcal}

\newcommand{\G}{\mc{G}}
\newcommand{\T}{\mc{T}}
\newcommand{\X}{\mc{X}}

\newcommand{\bqp}{\text{MBQP}}

\newcommand{\sdp}{\text{SDP}}
\newcommand{\cpp}{\text{CPP}}

\newcommand{\uc}{\mathcal{UC}}

\renewcommand{\a}{\mathbf{a}}
\newcommand{\aphi}{\a}
\newcommand{\aphitop}{\a^{\top}}

\newcommand{\x}{\mathbf{x}}

\newcommand{\opt}{\text{opt}}
\newcommand{\rel}{\text{rel}}
\newcommand{\feas}{\text{Feas}}
\newcommand{\tr}{\text{Tr}}

\newcommand{\pmax}{P^{\max}_g}
\newcommand{\pmin}{P^{\min}_g}

\usepackage{placeins}
\newcommand{\minup}{\Delta^{\rm{Up}}_{g}}
\newcommand{\mindn}{\Delta^{\rm{Down}}_{g}}
\newcommand{\rup}{R^{\rm{U}}_g}
\newcommand{\rdn}{R^{\rm{D}}_g}

\newcommand{\pmbqp}{\mc{P}^{\rm{MBQP}}}
\newcommand{\pcpp}{\mc{P}^{\rm{CPP}}}
\newcommand{\psdp}{\mc{P}^{\rm{SDP}}}

\newcommand{\mbqp}{\rm{MBQP}}
\newcommand{\vrd}{\mathbf{r}^d}

\newcommand{\vo}{\textbf{0}}
\newcommand{\va}{\mathbf{a}}

\newcommand{\vc}{\textbf{c}}

\newcommand{\vr}{\mathbf{r}}
\newcommand{\vu}{\textbf{u}}
\newcommand{\vx}{\mathbf{x}}
\newcommand{\vh}{\mathbf{h}}

\newcommand{\vz}{\textbf{z}}
\renewcommand{\vs}{\mathbf{s}}
\newcommand{\bpm}{\begin{pmatrix}}
\newcommand{\epm}{\end{pmatrix}}

\newcommand{\N}{\mathbb{N}}

\newcommand{\ve}{\mathbf{e}}

\newcommand{\vrel}{V^{\rm{Rel}}}

\newcommand{\vD}{\mathbf{D}}
\newcommand{\zfeas}{z^{\rm{Feas}}}
\newcommand{\vxfeas}{\mathbf{x}^{\rm{Feas}}}
\newcommand{\vsfeas}{\mathbf{s}^{\rm{Feas}}}
\newcommand{\vrfeas}{\mathbf{r}^{\rm{Feas}}}
\renewcommand{\N}{\mc{N}}

\newcommand{\vmax}{V^{\max}_k}
\newcommand{\vmin}{V^{\min}_k}
\newcommand{\smax}{S^{\max}_{ij}}
\newcommand{\vv}{\mathbf{v}}
\newcommand{\vvr}{\mathbf{v}^R}
\newcommand{\vvi}{\mathbf{v}^I}
\renewcommand{\vr}{v^{R}}
\newcommand{\vi}{v^{I}}
\newcommand{\vWr}{\bm{W}^R}
\newcommand{\vWi}{\bm{W}^I}
\newcommand{\vWc}{\bm{W}^{C}}
\newcommand{\Wr}{W^R}
\newcommand{\Wi}{W^I}
\newcommand{\Wc}{W^C}
\newcommand{\Dr}{D^R}
\newcommand{\Di}{D^I}
\newcommand{\img}{\mathbf{j}}
\newcommand{\qmax}{Q^{\max}_g}
\newcommand{\qmin}{Q^{\min}_g}
\newcommand{\sdpbd}{\rm SDP-bd}
\newcommand{\sdpdual}{\rm SDP-dual}
\newcommand{\psdpdual}{\mc{P}^{\rm SDP-dual}}
\newcommand{\psdpbd}{\mc{P}^{\rm SDP-bd}}
\newcommand{\fd}{f^{\rm d}}
\newcommand{\feastar}{\feas^*}
\newcommand{\thetalb}{\underline{\theta}}
\newcommand{\thetaub}{\bar{\theta}}
\newcommand{\ucsdp}{\rm UC-SDP}
\newcommand{\ac}{\mc AC}

\newcommand{\ploss}{p^{\rm Loss}}

\newcommand{\dr}{\rm d-R}
\newcommand{\di}{\rm d-I}
\newcommand{\lambdar}{\lambda^R}
\newcommand{\lambdai}{\lambda^I}

\newcommand{\vuhat}{\hat{\mathbf{u}}}
\newcommand{\vw}{\mathbf{w}}

\newcommand{\Ddotkt}{D^{\cdot}_{kt}}
\renewcommand{\hom}{\rm Hom}
\newcommand{\homdual}{\rm Hom-dual}

\newcommand{\Atilde}{\tilde{\bm{A}}}
\newcommand{\Btilde}{\tilde{\bm{B}}}
\newcommand{\Gtilde}{\tilde{\bm{G}}}

\newcommand{\Ytilde}{\tilde{\bm{Y}}}
\newcommand{\kcone}{\mc{K}_{\rm con}}
\newcommand{\sigmatilde}{\tilde{\sigma}}
\newcommand{\kappatilde}{\tilde{\kappa}}
\newcommand{\phitilde}{\tilde{\phi}}
\newcommand{\Thetatilde}{\tilde{\Theta}}
\newcommand{\Omegatilde}{\tilde{\Omega}}
\newcommand{\mutilde}{\tilde{\mu}}
\newcommand{\vxtilde}{\tilde{\vx}}
\newcommand{\Xtilde}{\tilde{\bm{X}}}

\usepackage{stackengine}
\stackMath

\usepackage{multicol}
\allowdisplaybreaks[4]
\usepackage{soul}
\sethlcolor{yellow}
\def\st{{\rm s.t.}}
\usepackage{calc}
\usepackage{pifont}

\newcommand{\xmark}{\text{\ding{55}}}

\usepackage{xr}
\usepackage{placeins}
\usepackage{xcolor}
\usepackage{booktabs} 
\usepackage{multirow} 
\usepackage{placeins}
\usepackage[ruled]{algorithm2e} 

\usepackage{natbib}
 \bibpunct[, ]{(}{)}{,}{a}{}{,}%
 \def\bibfont{\small}%
 \def\bibsep{\smallskipamount}%
\TheoremsNumberedThrough     
\ECRepeatTheorems

\EquationsNumberedThrough    

\MANUSCRIPTNO{}

\begin{document}


\RUNTITLE{Pricing Discrete and Nonlinear Markets With Semidefinite Relaxations} 

\TITLE{Pricing Discrete and Nonlinear Markets With Semidefinite Relaxations}


\ARTICLEAUTHORS{%
\AUTHOR{Cheng Guo, Lauren Henderson}
\AFF{School of Mathematical and Statistical Sciences, Clemson University, Clemson, South Carolina\\
\EMAIL{cguo2@clemson.edu, lahende@clemson.edu}}
\AUTHOR{Ryan Cory-Wright}
\AFF{Department of Analytics, Marketing and Operations, Imperial Business School, London, UK\\
\EMAIL{r.cory-wright@imperial.ac.uk}} 
\AUTHOR{Boshi Yang}
\AFF{School of Mathematical and Statistical Sciences, Clemson University, Clemson, South Carolina\\
\EMAIL{boshiy@clemson.edu}
}}

\ABSTRACT{
Nonconvexities in markets with discrete decisions and nonlinear constraints make efficient pricing challenging, often necessitating subsidies. A prime example is the unit commitment (UC) problem in electricity markets, where costly subsidies are commonly required. We propose a new pricing scheme for nonconvex markets with both discreteness and nonlinearity, by convexifying nonconvex structures through a semidefinite programming (SDP) relaxation and deriving prices from the relaxation's dual variables. When the choice set is bounded, we establish strong duality for the SDP, which allows us to extend the envelope theorem to the value function of the relaxation. This extension yields a marginal price signal for demand, which we use as our pricing mechanism. We demonstrate that under certain conditions—for instance, when the relaxation's right hand sides are linear in demand—the resulting lost opportunity cost is bounded by the relaxation's optimality gap. This result highlights the importance of achieving tight relaxations. The proposed framework applies to nonconvex electricity market problems, including for both direct current and alternating current UC. Our numerical experiments indicate that the SDP relaxations are often tight, reinforcing the effectiveness of the proposed pricing scheme. Across a suite of IEEE benchmark instances, the lost opportunity cost under our pricing scheme is, on average, 46\% lower than that of the commonly used fixed-binary pricing scheme.
} 




\maketitle

%


\section{Introduction} \label{ch: intro}
In many real-world markets, market participants select their purchase or selling quantities from a well-behaved convex choice set. In such markets, resource-allocation problems can be formulated as tractable optimization problems for which strong duality holds. As a result, market design principles such as social-welfare maximization and incentive compatibility are well justified by the First and Second Welfare Theorems \citep{feldman2006welfare}. Specifically, the welfare theorems guarantee that \textit{if strong duality is attained}, then simultaneously solving a resource allocation problem and its dual yields efficient dispatching of participants. See \cite{ferris2022dynamic} for a contemporary discussion. However, there exist markets whose choice set is nonconvex due to, e.g., discrete decisions and nonlinear feasible constraints, breaking strong duality. Consequently, dual solutions to the resource-allocation optimization problem may be unattainable or fail to support efficient allocation, thereby complicating price formation. 

A classic example is the electricity market, where system operators solve a unit commitment (UC) problem for efficient dispatch \citep{liberopoulos2016critical}. UC introduces discrete generator commitment (i.e., on/off) decisions and is thus nonconvex. Such nonconvexity makes it difficult to obtain prices that support efficient dispatch, necessitating uplift payments that subsidize generators so they are willing to follow an efficient dispatch. Uplift payments are criticized for being discriminatory and thus potentially misalign incentives \citep{schiro2015convex, mays2019quasi}. Recently, challenges posed by nonconvexity have become even more prominent in electricity market design, as there is a growing interest in incorporating alternating current (AC) into market models. Today's electricity market relies on direct current UC (DCUC) in market operations, with linear power-flow constraints. On the other hand, nonlinear AC power flows better capture network physics, and recent advances in the solvability of realistic-scale AC optimal power flow (ACOPF) problems make their practical implementation more promising. Yet, there is relatively little work on market design for ACUC problems, which are nonconvex due to both discrete decisions and nonlinear physical constraints. Beyond UC, nonconvexity structures appear in markets where allocation combines fixed costs, indivisibilities, or nonlinear physical constraints, such as in airport time slot allocation \citep{rassenti1982combinatorial} and in fisheries markets \citep{milgrom2025walrasian}.

Our main contributions are as follows: Our work proposes a computationally tractable pricing scheme for nonconvex markets with both discreteness and nonlinearity. We convexify the nonconvex market-clearing problem via a tight semidefinite programming (SDP) relaxation, and we define prices from this relaxation's dual variables based on the envelope theorem. The incentives under this pricing scheme are measured by the lost opportunity cost (LOC), the profit shortfall relative to the best response at posted prices. We show that tighter relaxations yield lower LOC under certain conditions. Numerical experiments on the electricity market show that our SDP relaxation is tight, and that the LOC under our pricing scheme is on average 46\% lower than under the fixed-binary pricing used in practice. Our framework is sufficiently general to encompass both use cases in the electricity market, namely the currently implemented DCUC model and the emerging ACUC setting. 

Our pricing step is solvable in polynomial time in theory, as it solves an SDP \citep{vandenberghe1996semidefinite}. Moreover, in practice SDPs can be solved efficiently using solvers such as MOSEK \citep{aps2020mosek}. In contrast, existing pricing schemes, such as fixed-binary, convex-hull, and copositive duality pricing as reviewed below, require solving either mixed-integer linear programs (MILPs) or copositive programs in pricing steps, and are thus NP-hard. 

The rest of this paper is organized as follows: Section \ref{ch: backgrnd} introduces the SDP relaxations for mixed-integer quadratic programs and for DCUC. Section \ref{ch: pricing_sdp} develops the envelope theorem for SDP relaxations and, based on this result, proposes a pricing scheme for DCUC. Section \ref{ch: ac} extends the proposed pricing scheme for ACUC. Section \ref{ch: loc} derives a tight upper bound on LOC. Section \ref{sec:pricingscheme} presents numerical comparisons on optimality gap and LOC levels. Proofs not included in the paper are in Sections \ref{ch: proof}-\ref{ch: proof_3} of the Appendix.

\subsection{Literature Review}\label{ch: lit}
Our work relates to the following streams of research in the economics and operations literature: (i) pricing in markets with discrete decisions and nonconvex costs, specifically (ii) pricing schemes in electricity markets modeled by UC, (iii) pricing under nonlinear AC network constraints, and (iv) pricing schemes with performance guarantees in nonconvex electricity markets. In what follows, we provide a detailed review of these works. 

First, a substantial body of work explores pricing in markets with discrete decisions, where standard welfare-theoretic logic underlying marginal-cost pricing fails. In discrete exchange economies, agents' demand sets are typically nonconvex, which can lead to the non-existence of competitive equilibria under standard linear pricing \cite{gale1984equilibrium}. Related work characterizes how indivisibilities generate complementarities across items \citep{bevia1999buying}. A key finding is that efficient allocations can often be recovered when agents’ preferences satisfy discrete convexity conditions: under such structures, linear prices may support efficient allocations \citep{danilov2001discrete,baldwin2019understanding}. However, without these conditions, marginal pricing generally cannot clear the market while covering participants’ nonconvex costs \cite{liberopoulos2016critical}, necessitating alternative schemes. These issues are particularly evident in electricity markets with unit commitment constraints, leading to the specialized pricing schemes for DCUC reviewed below.

Second, various pricing schemes are employed to price DCUC in real-world electricity markets \citep[see][for reviews]{liberopoulos2016critical,azizan2020optimal,eldridge2021efficient}. 
The fixed-binary pricing scheme \citep{o2005efficient} is widely used by U.S. system operators, obtaining dual prices from a modified DCUC optimization problem, which is a linear optimization problem. Unfortunately, while easy to implement, this approach can violate individual rationality under linear prices, requiring potentially large uplift payments to make participants whole. Conversely, convex-hull pricing \citep{hogan2003minimum,gribik2007market} generally yields lower uplift payments but may not support an equilibrium without out-of-market payments or adjustments. Recently, leveraging an exact copositive programming reformulation of the DCUC optimization and its strong duality, \citet{guo2021copositive} propose a copositive duality pricing scheme that exhibits desirable properties, such as revenue adequacy, and supports market equilibrium when additional constraints are imposed on the formulation. However, copositive programs are computationally challenging; thus, this method is limited to very small DCUC instances. This observation directly motivates the SDP-based pricing scheme proposed in this work. 

Third, our work also studies markets with nonlinearity, particularly driven by the nonconvex physical constraints that model AC power flows in ACUC. As interest in this area has grown in recent years, there is an emerging literature on pricing schemes in electricity markets that utilize AC power flows. Some studies focus on ACOPF models of power markets, which are nonlinear and do not include discrete unit commitment decisions. For instance, \citet{garcia2020generalized} propose a generalized convex hull pricing scheme for such markets, while \citet{ndrio2022pricing} introduce two pricing schemes for ACOPF-based markets, respectively, via Karush-Kuhn-Tucker conditions of the nonconvex ACOPF model and SDP relaxations of ACOPF. \citet{bichler2023getting} demonstrate that incorporating AC power flow in pricing enhances price signals, reducing side payments and redispatch costs. {\cgreen \citet{romero2025linear} propose a pricing method that is based on a tractable cutting-plane algorithm for ACOPF, providing price signals close to those from the Jabr second-order cone relaxation.} Notably, literature on pricing schemes for ACUC is more recent. To the best of our knowledge, \citet{garcia2024average} propose the only existing scheme, namely the average incremental cost pricing. Extending DCUC pricing schemes to ACUC could be nontrivial, e.g., extending the convex hull pricing would require solving partial Lagrangians of mixed-integer nonlinear programs (MINLP). In contrast, we propose an SDP-based pricing framework that is straightforward to implement and applies to both DCUC and ACUC.  


Finally, the literature on performance guarantees for pricing schemes in nonconvex electricity markets is relatively sparse. \citet{gribik2007market} establishes that the LOC under convex hull pricing equals the optimality gap of the partial Lagrangian relaxation. Additionally, \citet{guo2021copositive} derives a bound for the subsidy needed to cover the cost of non-individually rational generators under copositive duality pricing. In this work, we provide an upper bound on LOC, which applies to prices defined as subgradients of the value function of any relaxed UC formulation.
\subsection{Notation}
We use non-boldface characters, such as $b$, to denote scalars, lowercase boldface characters, such as $\bm{x}$, to denote vectors, uppercase boldface characters, such as $\bm{X}$, to denote matrices, and calligraphic uppercase characters, such as $\mathcal{Z}$, to denote sets. We let $[n]$ denote the set of running indices $\{1, ..., n\}$. We let $\mathbf{e}$ denote a vector of all $1$'s, $\bm{0}$ denote a vector of all $0$'s, $(\cdot)^c$ denote the complex conjugate (e.g., $(4+6\img)^c = 4-6\img$), and $\mathbb{I}$ denote the identity matrix. We denote an optimal solution of a variable with $*$. For example, an optimal solution of variable $p_{gt}$ is denoted as $p_{gt}^*$. We also employ a variety of matrix operators. We let 
$\mathcal{S}^n$ denote the $n \times n$ cone of symmetric matrices, $\mathcal{S}_n^+$ denote the $n \times n$ positive semidefinite cone, $\mc{C}_n^+=\{\bm{C} \in \mathcal{S}^n: \bm{C} = \sum_{k\in\mc{K}}\vz^k (\vz^k)^\top \text{for some finite}  \{\vz^k\}_{k\in\mc{K}}\subset \R_+^n\sm\{\vo\}\}\cup\{\vo\}$ denote the $n \times n$ completely positive cone, $\mathcal{C}_n=\{\bm{C} \in \mathcal{S}^n: \bm{x}^\top \bm{C}\bm{x} \geq 0 \ \forall \bm{x} \geq \bm{0}\}$ denote the copositive cone and $DNN_n=\mc{S}_n^+ \cap \mathbb{R}^{n \times n}_{+} \supseteq \mc{C}_n^+$ denote the doubly non-negative cone.
\section{Background}\label{ch: backgrnd}
    In this section, we provide background on semidefinite relaxations of mixed-binary quadratic problems (MBQPs) in Section \ref{ssec:mqqps}, and a semidefinite relaxation of DCUC in Section \ref{ssec:uc_shifted_factor}. This background facilitates the design of our UC pricing scheme in Section \ref{ch: pricing_sdp} and Section \ref{ch: ac}.
    \subsection{Semidefinite Relaxations of Mixed Binary Quadratic Problems}\label{ssec:mqqps}
    In this section, we review semidefinite relaxations for MBQPs of the form:
    \begin{equation}  
    \begin{aligned}
    \pmbqp : \quad\min \quad& \bm{x}^\top \bm{Q}\bm{x}+\vc^\top \bm{x}\label{eq: mbqp}\\ 
    \st \quad
    &\va^\top_{j} \vx = b_{j} && \forall j= 1,...,m \\
    & \vx\in\mb{R}^{n}_+\\
    & x_i \in\{0,1\}&~~&\forall i \in \mathcal{B},
    \end{aligned}
    \end{equation}
    where $\vx$ is a vector of nonnegative decision variables, $\bm{Q}\succeq \bm{0}$ is a positive semidefinite (PSD) objective matrix, and $ \mathcal{B} \subseteq [n]$ denotes the indices of the binary variables $x_i\in \{0, 1\}$. Note that signed variables $x$ can be modeled by Problem \eqref{eq: mbqp} via the standard decomposition $x=x_+ - x_-: x_+, x_- \geq 0$. 
    
    As shown by \citet{burer2009copositive}, MBQPs of the form $\pmbqp$ admit completely positive reformulations of the form:
    \begin{equation}\label{eq:cpp}
    \begin{aligned}
    \pcpp: \quad\min \quad& \tr(\bm{Q}^{\top} \bm{X})+\vc^\top \bm{x}\\ 
    \st \quad
    &\va_{j}^\top \vx = b_{j} && \forall j= 1,...,m\\ 
    & \tr(\va_{j} \va^\top_{j} \bm{X}) = b_{j}^2&~~&\forall j= 1,...,m \\
    & x_i = X_{ii} &&\forall i \in \mathcal{B}\\
     & \bm{Y} \in \mc{C}^+_{n+1},
    \end{aligned}
    \end{equation}
    where, to simplify notation, we write
    \begin{align}\label{eq: y}
    \bm{Y}: = \begin{bmatrix} 1 &\vx^\top\\ \vx & \bm{X}\end{bmatrix}
    \end{align}
    throughout the rest of the paper.
    
    In $\pcpp$, $\bm{Y}$ is a completely positive matrix where $\bm{X}$ models the non-convex outer product $\bm{xx}^\top$. As proven by \citet[Theorem 2.6]{burer2009copositive}, a rank-one constraint can be imposed in \eqref{eq:cpp} without loss of optimality, and in some rank-one solutions, $\bm{X}=\vx\vx^\top$ and $\vx$ has binary components. Thus, Problems \eqref{eq: mbqp}--\eqref{eq:cpp} satisfy \textit{objective value equivalence} in the sense that their optimal values coincide, and all optimal solutions to \eqref{eq: mbqp} solve \eqref{eq:cpp}. Moreover, \eqref{eq:cpp} and its dual satisfy strong duality under mild conditions such as bounded feasible region or a linear objective function in \eqref{eq: mbqp} \citep{brown2024copositive, cifuentes2024sensitivity}, which makes \eqref{eq:cpp} a theoretically appealing device for designing pricing mechanisms for problems with discrete decisions \citep{guo2021copositive}.  
    
    Note that the second set of constraints, $\tr(\va_{j} \va^\top_{j} \bm{X}) = b_{j}^2$, are known as reformulation linearization technique (RLT) constraints \citep{sherali1998exploiting}. They are obtained by squaring the left-hand and right-hand sides of $\va_{j}^\top \vx = b_{j}$, and then replacing the $\vx\vx^\top$ term with $\bm{X}$.

    
    Unfortunately, no solvers currently solve CPPs to optimality for non-trivial instances. 
    In the literature, CPPs are usually solved approximately via their SDP relaxations, such as the doubly non-negative relaxation:
    \begin{subequations}\label{eq: sdp}
    \begin{align}
    \psdp: \quad\min \quad& \tr(\bm{Q}^{\top} \bm{X})+\vc^\top \bm{x}\label{eq: sdp_1}\\ 
    \st \quad
    &\va_{j}^\top \vx = b_{j} && \forall j= 1,...,m && (\gamma_j)\label{eq: sdp_2}\\ 
    & \tr(\va_{j} \va^\top_{j} \bm{X}) = b_{j}^2&~~&\forall j= 1,...,m && (\omega_j)\label{eq: sdp_3}\\%
    & x_i = X_{ii} &&\forall i \in \mathcal{B} && (\kappa_i)\label{eq: sdp_4}\\%
    &\bm{Y} \in\mc{S}_{n+1}^+ &&  && (\Omega)\label{eq: sdp_5}\\
    & \bm{Y} \geq 0 && &&(\Theta),\label{eq: sdp_6}%
    \end{align}
    \end{subequations}
    where $\bm{Y} = \begin{pmatrix} 1 &\vx^\top\\ \vx & \bm{X}\end{pmatrix}$. The last two constraints ensure that $\bm{Y}$ is doubly nonnegative (DNN), i.e., positive semidefinite and elementwise nonnegative. This is a well-known relaxation of the constraint $\bm{Y} \in \mc{C}^+_{n+1}$ \citep[][]{diananda1962non, parrilo2000structured}. The Greek letters in brackets denote the dual variables for each constraint.

    \subsection{Semidefinite Relaxation of DCUC Problems}\label{ssec:uc_shifted_factor}
    In this section, we describe an MILP model and its SDP relaxation for DCUC problems. We present a {\it shift factor} formulation \eqref{eq: uc_model_sf} for DCUC problems \citep{fattahi2017conic}, which we denote by $\uc$. This formulation is equivalent to the commonly-used network flow UC formulation \citep{bienstock2024risk} by eliminating the voltage angle variables. 
    In numerical experiments, we observe that the SDP relaxation for \eqref{eq: uc_model_sf} is more efficient than the relaxation for the network flow formulation.
    We note, however, that our pricing framework can be easily extended to other UC formulations. 
    
    \begin{subequations}\label{eq: uc_model_sf}
    \begin{align}
    \min \quad& \sum_{g\in\mc{G}}\sum_{t\in\mc{T}} \left(C^p_g p_{gt} + C^{u}_g u_{gt} + C^{v}_g v_{gt} + C_g^z z_{gt}\right)\label{eq: uc_model_sf_0}\\
    \st \quad & \sum_{g\in\mc{G}} p_{gt} = \sum_{k\in\mc{N}}D_{kt} &&\forall t\in\mc{T} &&(\lambda_t)\label{eq: uc_model_sf_1}\\
    & B_{ij} \sum_{k=1}^{|\mc{N}|}\left(\left(S_{ik} - S_{jk}\right)(\sum_{g\in\mc{G}_k}p_{gt}-D_{kt})\right) \geq P^{\text{Trans},\min}_{ij} &&\forall (i, j)\in\mc{L}, t\in\mc{T} &&(\xi^{\min}_{ijt})\label{eq: uc_model_sf_2}\\
    & B_{ij} \sum_{k=1}^{|\mc{N}|}\left(\left(S_{ik} - S_{jk}\right)(\sum_{g\in\mc{G}_k}p_{gt}-D_{kt})\right) \leq P^{\text{Trans},\max}_{ij}  &&\forall (i, j)\in\mc{L}, t\in\mc{T}  &&(\xi^{\max}_{ijt})\label{eq: uc_model_sf_3}\\
    & p_{gt} \geq \pmin z_{gt}  &&\forall g\in\mc{G}, t\in\mc{T} && \label{eq: uc_model_sf_4}\\
    & p_{gt} \leq \pmax z_{gt} &&\forall g\in\mc{G}, t\in\mc{T}  && \label{eq: uc_model_sf_5}\\
    & u_{gt}-v_{gt} = z_{gt} - z_{g, t-1}  &&\forall g\in\mc{G}, t\in\mc{T}\sm\{1\}  &&\label{eq: uc_model_sf_7} \\
    & p_{gt} - p_{g,t-1}\leq \rup z_{g,t-1}+\pmin u_{gt} &&\forall g\in\mc{G}, t\in\mc{T}\sm\{1\} &&\label{eq: uc_model_sf_8}\\
    & p_{g,t-1}-p_{gt}\leq \rdn z_{gt}+\pmin v_{gt} &&\forall g\in\mc{G}, t\in\mc{T}\sm\{1\} &&\label{eq: uc_model_sf_9}\\
    & \sum_{\tau=\max(1,t-\minup)}^{t} u_{g\tau} \leq z_{gt} &&\forall g\in\mc{G}, t\in\mc{T}\sm\{1\} &&\label{eq: uc_model_sf_10}\\
    & \sum_{\tau=\max(1,t-\mindn)}^{t} v_{g\tau} \leq 1-z_{gt} && \forall g\in\mc{G}, t\in\mc{T}\sm\{1\} &&\label{eq: uc_model_sf_11} \\
    & u_{gt}, v_{gt}, z_{gt} \in \{0,1\} && \forall g\in\mc{G}, t\in\mc{T} &&   \label{eq: uc_model_sf_12bin} \\
    & p_{gt}, u_{gt}, v_{gt}, z_{gt}\geq 0 &&\forall g\in\mc{G}, t\in\mc{T}\label{eq: uc_model_sf_13} 
    , 
    \end{align}
    \end{subequations}
    where $\mc{N}$ denotes the set of buses (nodes), $\mc{L}$ denotes the set of transmission lines, $\mc{T}$ denotes the set of time periods, $\mc{G}$ denotes the set of generators, and $\mc{G}_k$ denotes the set of generators at bus $k\in\mc{N}$. For the variables, $p_{gt}$ denotes the electricity production of generator $g$ during time period $t$. The binary variable $z_{gt}$ denotes the on or off status, which equals 1 if $g$ is on during time period $t$ and 0 otherwise. The binary variable $u_{gt}$ denotes the startup decision, which equals 1 if generator $g$ is started up during time period $t$ and equals 0 otherwise. The shutdown decision variable $v_{gt}$ is similarly defined. The Greek letters in the rightmost column denote the corresponding dual variables for the linear or semidefinite relaxation of \eqref{eq: uc_model_sf}. 

    The objective \eqref{eq: uc_model_sf_0} minimizes the total cost, which includes the production cost $C^p_g p_{gt}$, the startup cost $C^{u}_g u_{gt}$, the shutdown cost $C^{v}_g v_{gt}$, and the fixed cost of keeping a generator on ($C_g^z z_{gt}$). 
    Constraints \eqref{eq: uc_model_sf_1} are the market-clearing constraints, which ensure that the total demand equals the total production at each time period. Constraints \eqref{eq: uc_model_sf_2} and \eqref{eq: uc_model_sf_3} are respectively the lower and upper bounds for the flow in line $(i, j)$. The left-hand sides of both constraints represent the flow in $(i, j)$, where $B_{ij}$ is the susceptance of line $(i,j)$. 
    Constraints \eqref{eq: uc_model_sf_4} and \eqref{eq: uc_model_sf_5} bound production levels, where $\pmin$ and $\pmax$ are respectively the minimum and maximum production levels of $g$. Constraints \eqref{eq: uc_model_sf_7} connects $u_{gt}, v_{gt}$, and $z_{gt}$. Constraints \eqref{eq: uc_model_sf_8} and \eqref{eq: uc_model_sf_9} respectively bound the rate of ramping up and down, where $\rup$ and $\rdn$ respectively denote the ramping up and down limits. Constraints \eqref{eq: uc_model_sf_10} and \eqref{eq: uc_model_sf_11} respectively enforce the minimum up and down time constraints, where $\minup$ and $\mindn$ respectively denote the minimum up and down time.

    $\uc$ is an MILP, which is a special type of MIQP. We obtain the SDP relaxation of $\uc$ in the form of $\mc{P}^\sdp$ and denote this relaxation problem \eqref{eq: uc_sdp} as $\uc^\sdp$. To formulate $\uc^\sdp$, we add slack variables to inequality constraints of $\uc$ to make them equalities. The $\uc^\sdp$ formulation is as follows: 
    \allowdisplaybreaks
    \begin{subequations}\label{eq: uc_sdp}
    \begin{align}
    \min \quad& \sum_{g\in\mc{G}}\sum_{t\in\mc{T}} \left(C^p_g p_{gt} + C^{u}_g u_{gt} + C^{v}_g v_{gt} + C_g^z z_{gt}\right)\hspace{-3cm}\label{eq: uc_sdp_0}\\
    \st \quad& \sum_{g\in\mc{G}} p_{gt} = \sum_{k\in\mc{N}}D_{kt} && \forall t\in\mc{T} \label{eq: uc_sdp_1}\\
    & \va_{1ijt}^\top\vx = B_{ij} \sum_{k=1}^{|\mc{N}|}\left(S_{ik} - S_{jk}\right)D_{kt}+P^{\text{Trans},\min}_{ij}&&\forall (i, j)\in\mc{L}, t\in\mc{T}\label{eq: uc_sdp_2}\\
    & \va_{2ijt}^\top\vx = B_{ij} \sum_{k=1}^{|\mc{N}|}\left(S_{ik} - S_{jk}\right)D_{kt}+P^{\text{Trans},\max}_{ij} &&\forall (i, j)\in\mc{L}, t\in\mc{T}\label{eq: uc_sdp_3}\\
    &\aphitop_{egt} \vx = b_{egt} && \forall e= 1,\dots,m; g\in\mc{G}, t\in\mc{T}\label{eq: uc_sdp_4}\\
    &\tr(\vh_t \vh^\top_t \bm{X}) = \left(\sum_{k\in\mc{N}}D_{kt}\right)^2&&\forall t\in\mc{T}\label{eq: uc_sdp_5}\\
    &\tr(\aphi_{1ijt} \aphitop_{1ijt} \bm{X}) = \Large(B_{ij} \sum_{k=1}^{|\mc{N}|}\left(S_{ik} - S_{jk}\right)D_{kt}\\ &+ P^{\text{Trans},\min}_{ij}\Large)^2  &&\forall (i, j)\in\mc{L}, t\in\mc{T}\label{eq: uc_sdp_6}\\
    &\tr(\aphi_{2ijt} \aphitop_{2ijt} \bm{X}) = \Large(B_{ij} \sum_{k=1}^{|\mc{N}|}\left(S_{ik} - S_{jk}\right)D_{kt} \\&+P^{\text{Trans},\max}_{ij}\Large)^2 \hspace{-0.3cm}&&\forall (i, j)\in\mc{L}, t\in\mc{T}\label{eq: uc_sdp_7}\\
    &\tr(\aphi_{egt} \aphitop_{egt} \bm{X}) = b_{egt}^2&&\forall e= 1,...,m; g\in\mc{G}, t\in\mc{T}\label{eq: uc_sdp_8}\\
    &x_i = X_{ii} &&\forall i\in\mc{B} \label{eq: uc_sdp_9}\\
    &\bm{Y}\in\mc{S}^+_{\dim(\vx)+1}, \bm{Y}\geq 0, \label{eq: uc_sdp_10}
    \end{align}
    \end{subequations}
    where $\vx\in\R^{\dim{\vx}}$ denotes the vector of all variables, including slack variables, of $\uc$. Note that $\bm{Y}: = \begin{bmatrix} 1 &\vx^\top\\ \vx & \bm{X}\end{bmatrix}$. \eqref{eq: uc_sdp_2} and \eqref{eq: uc_sdp_3} are equality versions of constraints \eqref{eq: uc_model_sf_2} and \eqref{eq: uc_model_sf_3}, with vectors $\va_{1ijt}$ and $\va_{2ijt}$ denoting their constraint coefficients. \eqref{eq: uc_sdp_4} represents generator operational constraints \eqref{eq: uc_model_sf_4} - \eqref{eq: uc_model_sf_11} plus the upper bound on binary variables, where $\va_{egt}$ and $b_{egt}$ respectively denote the constraint coefficient vectors and right-hand sides of those operational constraints. \eqref{eq: uc_sdp_5}-\eqref{eq: uc_sdp_8} are RLT constraints obtained by squaring \eqref{eq: uc_sdp_2}-\eqref{eq: uc_sdp_4}, where $\vh_t$ is the vector of coefficients for \eqref{eq: uc_sdp_1}, i.e., entries of $\vh_t$ equal 1 for $p_{gt},g\in\mc{G}$ and 0 otherwise. \eqref{eq: uc_sdp_9} corresponds to \eqref{eq: sdp_4} in $\mc{P}^\sdp$. \eqref{eq: uc_sdp_10} includes DNN constraints.

    \section{UC Pricing with Semidefinite Relaxation}\label{ch: pricing_sdp}
    In this section, we propose a pricing scheme for UC that leverages the envelope theorem from economics. We first develop an envelope theorem for a semidefinite relaxation of $\pmbqp$ in Section \ref{ch: env_th}. Then, we leverage this envelope theorem to define a pricing scheme for DCUC's semidefinite relaxation in Section \ref{ssec:envthm_sdppricing}. 
    \subsection{Envelope Theorem for Semidefinite Relaxation}\label{ch: env_th}
    The envelope theorem \citep[see][for an overview]{milgrom2002envelope} quantifies how the optimal value of an optimization problem changes when its parameters are perturbed, by characterizing the directional derivative of the optimization problem's value function.
    In this section, we use the envelope theorem to obtain derivatives of value functions for a strengthened SDP relaxation $\mc{P}^\sdpbd$, which is parameterized in the right-hand side. This allows us to analyze the sensitivity of $\pmbqp$ and, thus, to design pricing schemes for DCUC in Section \ref{ssec:envthm_sdppricing}.

    We use the following assumption throughout this paper:

    \begin{assumption}[Bounded continuous variables in MBQP]\label{ass: bound}
        The non-negative continuous variables in MBQP $\mc{P}^\mbqp$ (Section \ref{ssec:mqqps}) are bounded from above. In particular, for each continuous variable $x_i: i\in [n]\sm\mc{B}$, we have the finite bound $x_i \le U_i$ for some $U_i$. 
    \end{assumption}

    Assumption \ref{ass: bound} is standard for real-world UC problems: continuous decision variables represent physical quantities with explicit engineering limits. For instance, the continuous variable for power output $p_{gt}$ is bounded by the generator's bid-based capacity limit $\pmax$, as indicated by the constraint \eqref{eq: uc_model_sf_5}, often coupled with ramping and reserve limits. Similarly, AC power-flow equations (see Section \ref{ch: ac}) incorporate engineering constraints, such as bounds on active/reactive injections and voltage magnitudes. Boundedness is also a common assumption in copositive optimization to ensure strong duality \citep[e.g.,][]{brown2024copositive,cifuentes2024sensitivity}. This assumption is necessary for some of our theoretical results, which require strong duality of SDP.  

    Under Assumption \ref{ass: bound}, we strengthen our SDP relaxation by imposing the valid inequalities $X_{ii}\leq U_i^2$ for each continuous variable index $i\in [n]\sm\mc{B}$ (note that this corresponds to $x_i^2 \leq U_i^2$ in an exact rank-one lifting where $\bm{X}=\bm{x}\bm{x}^\top$). This yields the following strengthened SDP relaxation, in which we introduce a scalar parameter $\theta$ that appears in the right-hand sides of the constraints, and introduce the dummy matrix $\bm{Y} = \begin{pmatrix} 1 &\vx^\top\\ \vx & \bm{X}\end{pmatrix}$ for conciseness:
    \begin{subequations}\label{eq:sdp-bd}
    \begin{align}
    \psdpbd(\theta): \quad\min \quad& \tr(\bm{Q}^{\top} \bm{X})+\vc^\top \bm{x}\label{eq: sdp_bd}\\ 
    \st \quad
    &\va_{j}^\top \vx = b_{j}(\theta) && \forall j= 1,...,m &&(\gamma_j) \label{eq: sdp_bd4}\\ 
    & \tr(\va_{j} \va^\top_{j} \bm{X}) = b_{j}(\theta)^2&~~&\forall j= 1,...,m && (\omega_j)\label{eq: sdp_bd4}\\
    & X_{ii} \leq  U_i^2 &&\forall i \in [n]\sm\mathcal{B} && (\phi_i)\label{eq: sdp_bd3}\\
    & \eqref{eq: sdp_4}-\eqref{eq: sdp_6},
    \end{align}
    \end{subequations}
    where $b_{j}(\theta), j= 1,...,m$ are differentiable functions of $\theta$. {\cgreen Similarly, we write the parameterized $\mc{P}^\mbqp$ with $b_j(\theta)$ in the right-hand side as $\mc{P}^\mbqp(\theta)$.}

    We now define the value function of Problem \eqref{eq:sdp-bd}: 
    \begin{definition}[Primal value function of $\psdpbd(\theta)$]
    Let $\feas(\cdot)$  denote the feasible region. Define the value function 
    \begin{align}
         V^{\sdp}(\theta) := \min_{\bm{Y}\in\feas(\mc{P}^{\sdpbd}(\theta))} \tr(\bm{Q}^{\top}\bm{X})+\vc^\top \bm{x}, 
    \end{align}
    \end{definition}

    To derive an envelope theorem for the strengthened SDP relaxation \eqref{eq:sdp-bd}, we first prove that its feasible region is nonempty, bounded, and consists of two cones and a hyperplane. Such conditions are standard for ensuring strong duality in conic optimization \citep{kim2025strong, brown2024copositive}:

    \begin{proposition}[Strong duality for {\cgreen $\mc{P}^\sdpbd(\theta)$}]\label{th: dual}
        If {\cgreen $\mc{P}^\mbqp(\theta)$} is feasible and satisfies Assumption \ref{ass: bound}, then strong duality holds between the SDP relaxation {\cgreen $\mc{P}^\sdpbd(\theta)$} and its dual {\cgreen $\mc{P}^{\sdpdual}(\theta)$}.
    \end{proposition}
    Proposition \ref{th: dual} is arguably non-trivial, because there exist families of SDPs with bounded feasible regions and strictly positive duality gaps \cite[Example 1]{pataki2018positive}. Our result essentially follows by showing that {\cgreen $\mc{P}^{\sdpbd}(\theta)$ is feasible and has a compact feasible region}, then leveraging a result due to \citet{kim2025strong} which guarantees that for problems of the form {\cgreen $\mc{P}^{\sdpbd}(\theta)$}, feasibility and primal compactness give strong duality. 

    With strong duality, the primal and dual value functions are equivalent. We write the dual value function of $\psdpbd(\theta)$ explicitly because it is used in the envelope theorem.

    \begin{definition}[Dual value function of $\psdpbd(\theta)$]
        Let $\Psi:= (\pmb \gamma, \pmb \omega, \pmb\phi, \pmb \kappa, \bm{\Omega}, \bm{\Theta})$ denote the tuple of variables in $\mc{P}^{\sdpdual}$. Define 
        \begin{align}
        V^{\sdpdual}(\theta) := \max_{\Psi\in\feas(\psdpdual(\theta))} \fd(\Psi, \theta),
        \end{align}
        where $\fd(\Psi, \theta) := \sum_{j=1}^m \gamma_j b_j(\theta) + \omega_j b_j (\theta)^2+\sum_{i\in [n]\sm\mc{B}} U_i^2 \phi_i$. Further, let $\feastar(\psdpdual(\theta))$ denote the set of optimal dual solutions at $\theta$.
    \end{definition}
    
    We now present an envelope theorem for the tightened SDP relaxation. The result essentially follows by combining \citep[Theorem 1]{milgrom2002envelope} with the strong duality of the SDP relaxation (Proposition \ref{th: dual}):

    \begin{theorem}[Envelope theorem for $\mc{P}^\sdpbd$]\label{prop:envelopetheorem1}
    Let $[\thetalb, \thetaub]$ be an interval where $\psdpbd(\theta)$ is feasible for every $\theta\in [\thetalb, \thetaub]$ and $\Psi^*\in\feastar(\psdpdual(\theta))$. 
    
    If $\theta>\thetalb$ and $V^\sdp(\theta)$ is left differentiable, then $V^{\sdp'}(\theta-) \le \fd_{\theta}(\Psi^*, \theta)$. If $\theta<\thetaub$ and $V^\sdp(\theta)$ is right differentiable, then $V^{\sdp'}(\theta+) \ge \fd_{\theta}(\Psi^*, \theta)$.  If $\theta\in[\thetalb, \thetaub]$ and $V^\sdp(\theta)$ is differentiable, then $V^{\sdp'}(\theta) = \fd_{\theta}(\Psi^*, \theta)$. 
        
    \end{theorem}

   Therefore, for any right-hand side parameter, its marginal effect on the optimal objective value at differentiable points is determined by the optimal dual solution of the strengthened SDP relaxation. 

   Theorem \ref{prop:envelopetheorem1} can immediately be extended to provide an envelope theorem as multiple parameters vary simultaneously, a result that will be useful when deriving prices for UC:

    \begin{corollary}[Vector envelope theorem for $\psdpbd$]\label{cor:vector-envelope} Let $\psdpbd(\vu)$ be a version of $\psdpbd$ whose right-hand sides $b_j(\vu), j=1,...,m$ are differentiable functions of vector $\vu\in \mathbb R^d$. Let $\mc{U}\subseteq\mathbb R^d$ be a set where $\psdpbd(\vu)$ is feasible for every $\vu\in\mc{U}$ and $\Psi^*\in \mathrm{Feas}^*(P_{\mathrm{SDP\text{-}dual}}(\vu))$.
    Let $V(\vu)$ denote its value function. Then $V(\vu)=\max_{\Psi\in\mathrm{Feas}(P^{\mathrm{SDP\text{-}dual}}(\vu))} f^d(\Psi,\vu)$. 
    
    If $V(\vu)$ is differentiable at $\vu$, then $\nabla V(\vu)=\nabla_{\vu} f^d(\Psi^*,\vu)$.
\end{corollary}

    \subsection{Envelope Theorem for DCUC Pricing}\label{ssec:envthm_sdppricing}
    We consider the pricing problem in the day-ahead wholesale electricity market. In this market, the system operator clears the market by solving a DCUC problem that minimizes total operational cost. Since the demand is inelastic, this is equivalent to maximizing social welfare. We use the envelope theorem to identify how variations in demand change social welfare at the margin, and to treat those marginal effects as prices. 

    Because DCUC is an MILP, it generally does not have a dual problem for obtaining shadow prices. Our approach is to apply the envelope theorem on the value function of the DCUC SDP relaxation, and interpret the resulting derivatives as prices.

    \begin{definition}[Strengthened DCUC SDP relaxation and its dual]
        We consider the strengthened DCUC SDP relaxation $\uc^\sdpbd$, which relaxes the DCUC model $\uc$ following $\psdpbd$. That is to say, $\uc^\sdpbd$ strengthens $\uc^\sdp$ (Section \ref{ssec:uc_shifted_factor}) by imposing bounds $X_{ii}\le U_i^2$ for all continuous variables. Define the value function of $\uc^\sdpbd$ parameterized by demand $D_{kt}$ as $V^{\ucsdp}(D_{kt})$. 
        
        We denote the dual problem of $\uc^\sdpbd$ as $\uc^\sdpdual$. Let $f^{\rm uc-d}(\Psi, D_{kt})$ denote the objective of {\cgreen $\uc^\sdpdual$} parameterized by $D_{kt}$. Note that with a slight abuse of notation, we use $\Psi$ to denote the tuple of variables for both the general SDP dual and its UC counterpart. The intended meaning is clear from the context.
    \end{definition}

    To apply the envelope theorem, consider a feasible strengthened DCUC SDP relaxation parameterized by demand $D_{kt}$. We obtain a partial derivative of its dual problem's objective function at an optimal solution with respect to $D_{kt}$, denoted as $r^d_{kt} := f^{\rm uc-d}_{D_{kt}}(\Psi^*, D_{kt})$. We use $r^d_{kt}$ as the electricity price at bus $k$ and time $t$, and it represents the marginal value of supplying demand. The following proposition provides an expression for $r^d_{kt}$.

    \begin{proposition}[Expression for DCUC pricing]\label{th: foc_d}
         Let $\hat{\lambda}_t$ be the dual variable associated with \eqref{eq: uc_sdp_5}. Similarly, $\hat{\xi}^{\min}_{ijt}$ and $\hat{\xi}^{\max}_{ijt}$ are the dual variables of \eqref{eq: uc_sdp_6} and \eqref{eq: uc_sdp_7}, respectively. The price $r^d_{kt}$ is then given by the expression:
        \begin{align}\label{eq: foc_d}
            r^d_{kt} = &\lambda^*_t + 2\hat{\lambda}^*_t\sum_{k'\in\mathcal{N}}D_{k't} + \sum_{(i,j) \in \mathcal{L}} B_{ij}(S_{ik} - S_{jk})(\xi^{\min*}_{ijt} + \xi^{\max*}_{ijt})\nonumber\\
            & + 2\sum_{(i,j) \in \mathcal{L}}B_{ij}^2(S_{ik} - S_{jk})\Big(\sum_{k' \in \mathcal{N}}(S_{ik'} - S_{jk'})D_{k't} \Big)(\hat{\xi}^{\min*}_{ijt} + \hat{\xi}^{\max*}_{ijt}) \nonumber\\
            & + 2\sum_{(i,j) \in \mathcal{L}}B_{ij}(S_{ik} - S_{jk})(P^{\text{Trans},\min}_{ij}\hat{\xi}^{\min*}_{ijt} + P^{\text{Trans},\max}_{ij}\hat{\xi}^{\max*}_{ijt}).
        \end{align}
    \end{proposition}

    Note that for ease of exposition, we treat the value function of the strengthened DCUC SDP relaxation as a function of a single parameter $D_{kt}$ when deriving prices in Sections \ref{ch: pricing_sdp} and \ref{ch: ac}. The same derivation extends to vector-valued demand parameters, as indicated by Corollary \ref{cor:vector-envelope}. This vector perspective is used in Section \ref{ch: loc}, where the value function is written in terms of a demand vector.

    In addition to the price $r^d_{kt}$, each generator also receives an uplift payment for its LOC \cite{azizan2020optimal}, which equals the extra profit it earns if it maximizes its own profit instead of following the optimal dispatch decisions of $\uc$, given the price vector $\vrd=[r^d_{kt}]_{k\in\N, t\in\T}$. We denote the LOC of $g$ as $U_g(\vrd)$. LOC essentially measures the incentive deviation under posted prices. A detailed discussion of LOC is included in Section \ref{ch: loc}, where a bound is provided for the total LOC that the system operator pays. 

    Formally, we define the SDP-based mechanism for DCUC as follows:
    
    \begin{definition}[SDP-based mechanism for DCUC] \label{dcuc_pricing}
        Under the SDP-based mechanism DCUC, the system operator follows optimal dispatch decisions of $\uc$. Prices are obtained via solving the SDP relaxation $\uc^\sdpbd$. The system operator:
        \BI
        \I pays $\pi^G_{gt} = r^d_{k(g),t} p^*_{gt}$ to generator $g$ at each time $t\in\T$, where $k(g)$ denotes the node at which $g$ is located, and $p^*_{gt}$ is the optimal production level from the MILP $\uc$. In addition, the system operator pays an LOC uplift payment $U_g(\vrd)$ to generator $g$; and
        \I collects $\pi^L_t = \sum_{k\in\mc{N}}r^d_{kt} D_{kt}$ from the load at each time $t\in\T$. In addition, the system operator collects a price adder of $\frac{\sum_{g\in\G}U_g(\vrd)}{\sum_{k\in\mc{N}, t\in\T }D_{kt}}$ per unit of demand.
        \EI 
    \end{definition}
Observe that Definition \ref{dcuc_pricing}'s pricing mechanism is cost recovering (i.e., ensures that all generators recover their variable and fixed costs) by construction. Moreover, since the demand is inelastic, the price adder does not affect consumption levels. 

    \begin{remark}
        In instances where there are multiple dual-optimal solutions to $\uc^\sdpbd$, one can use any dual-optimal solution for pricing. Thus, in our numerical experiments, we use the dual solutions provided by the SDP solver without checking for uniqueness. One could also use the lexicographically smallest prices by solving a second optimization problem, analogous to the Pareto-optimal Benders cuts of \citet{magnanti1981accelerating}.
    \end{remark}

    Since the SDP-based price measures marginal effects of local demand at differentiable points of the value function $V^{\ucsdp}(D_{kt})$, it is natural to ask when the value function is differentiable. {\cgreen Proposition \ref{th: continuous}} shows that $V^{\ucsdp}(D_{kt})$ is differentiable everywhere except possibly at finitely many points. {\cgreen It is proved by showing that $V^{\ucsdp}(D_{kt})$ is a semialgebraic function, and thus is nondifferentiable at finitely many points due to the monotonicity lemma \citep[Lemma 2.3]{lee2017generic}.}

    \begin{proposition}[{\cgreen Finite nondifferentiability} of $V^{\ucsdp}(D_{kt})$]\label{th: continuous}
    Fix all demand entries except $D_{kt}$, and let $[\underline{D}_{kt}, \overline{D}_{kt}]$ be any compact interval such that the strengthened SDP relaxation $\uc^\sdpbd(D_{kt})$ is feasible for every $D_{kt}\in [\underline{D}_{kt}, \overline{D}_{kt}]$. Then, the value function $V^{\ucsdp}(D_{kt})$ is differentiable everywhere on $[\underline{D}_{kt}, \overline{D}_{kt}]$ except possibly at finitely many points. 
    \end{proposition}
    \begin{proof}{Proof of Proposition \ref{th: continuous}}
    Let $\bm{w}$ denote the full vector of primal decision variables in $\uc^\sdpbd(D_{kt})$, including all lifted matrix entries (e.g., taking $\mathrm{Vec}(\bm{X})$ where appropriate), and rewrite the objective function as $\vc^\top \bm{w}$. Then, by Assumption \ref{ass: bound}, all continuous variables in $\bm{w}$ are contained in a compact set, as shown in detail in the proof of Proposition \ref{th: dual}. Hence, since $V^{\mathrm{UC\text{-}SDP}}(D_{kt})$ is feasible for all $D_{kt}\in [\underline{D}_{kt}, \overline{D}_{kt}]$ and has a {\cgreen compact} feasible region, it attains a finite value for all $D_{kt}$ in the interval. 

    \noindent{\it Step 1: } We show that $V^{\mathrm{UC\text{-}SDP}}(D_{kt})$ is a semialgebraic function on $[\underline{D}_{kt}, \overline{D}_{kt}]$. We claim that the set
    \[
    \mathcal S \;:=\;\{(D_{kt},\bm{w},z)\in \R\times\R^p\times\R:\ D_{kt}\in [\underline{D}_{kt}, \overline{D}_{kt}], \bm{w}\in F(D_{kt}),\ z=\vc^\top\vw\}
    \]
    is semialgebraic. Indeed, for any fixed \(D_{kt}\), the feasible region \(F(D_{kt})\) of \(\uc^{\mathrm{SDP\text{-}bd}}(D_{kt})\) is the intersection of a spectrahedron and a polyhedron, with both being semialgebraic. Since semialgebraic sets are closed under finite intersections \citep[Section A.4.4]{blekherman2012semidefinite}, \(F(D_{kt})\) is also semialgebraic. In addition, $\underline D_{kt}\le D_{kt}\le \overline D_{kt}$ and $z=\vc^\top \vw$ are linear inequalities/equalities. It follows that \(\mc{S}\) is a semialgebraic subset of \(\mathbb{R}^{p+2}\).

    
    Let $\mathcal A\subseteq \R^2$ be the projection of $\mathcal S$ onto the $(D_{kt},z)$ coordinates. This set is also semialgebraic, as by the Tarski-Seidenberg theorem, semialgebraic sets are closed under projections \citep[Theorem A.49]{blekherman2012semidefinite}. Similarly, define
\[
\mathcal B \;:=\;\{(D_{kt},z)\in \R^2:\ D_{kt}\in [\underline{D}_{kt}, \overline{D}_{kt}], z= \vc^\top\vw, \st~\vw\in F(D_{kt}) \text{ and }\exists\,\hat{\bm{w}}\in F(D_{kt})\text{ with }\vc^\top\hat{\vw}<z\},
\]
which is the set of $(D_{kt},z)$ pairs where $z$ is not the optimal value at $D_{kt}$. The set $\mc{B}$ is also semialgebraic, since it equals a projection of $\tilde{\mc{B}}$, where $\tilde{\mc{B}} := \{(D_{kt},\bm{w},z)\in \R\times\R^p\times\R:\ D_{kt}\in [\underline{D}_{kt}, \overline{D}_{kt}], \bm{w}\in F(D_{kt}),\ z=\vc^\top\vw, \st~\exists\,\hat{\bm{w}}\in F(D_{kt})\text{ with }\vc^\top\hat{\vw}<z\}$ is a semialgebraic set.

Then, because the optimum of $\uc^\sdp$ is attainable in $[\underline{D}_{kt}, \overline{D}_{kt}]$, we have
\[
\mathrm{graph}(V^{\mathrm{UC\text{-}SDP}}(D_{kt}))\;=\;\mathcal A\setminus \mathcal B.
\]
Since semialgebraic sets are closed under set difference (as a result of being closed under finite intersection and complementation),
$\mathrm{graph}(V^{\mathrm{UC\text{-}SDP}}(D_{kt}))$ is semialgebraic. Therefore, \ $V^{\mathrm{UC\text{-}SDP}}(D_{kt})$ is a semialgebraic function of one real
variable on $[\underline{D}_{kt}, \overline{D}_{kt}]$. 

\noindent{\it Step 2:} To prove $V^{\mathrm{UC\text{-}SDP}}(D_{kt})$ is differentiable except possibly at the finitely many breakpoints, we {\cgreen use the monotonicity lemma \citep[Lemma 2.3]{lee2017generic}, which implies that for the semialgebraic function $V^{\mathrm{UC\text{-}SDP}}(D_{kt})$,} there exists a finite partition
$\underline D_{kt} = d_0 < d_1 < \cdots < d_N = \overline D_{kt}$
such that $V^{\mathrm{UC\text{-}SDP}}(D_{kt})$ is continuously differentiable on each open interval $(d_{i-1},d_i)$.
In particular, $V^{\mathrm{UC\text{-}SDP}}(D_{kt})$ is differentiable except possibly at the finitely many breakpoints
$\{d_1,\dots,d_{N-1}\}$, which proves the second claim. 
%
\qed
    \end{proof}

Importantly, the envelope theorem sensitivity bounds {\cgreen and finite nondifferentiability} underlying our pricing scheme do not rely on the SDP relaxation: they also hold for the original MBQP model via its exact CPP reformulation, since value functions of MBQP and CPP are equivalent and CPP has strong duality under Assumption \ref{ass: bound} \citep{brown2024copositive}. 

{\cgreen As a side note, value functions of MBQP and its SDP relaxation can also be continuous under additional structural conditions, such as sufficiently large value-of-lost-load penalties in $\uc$'s objective, which is common in unit commitment and ensures absolute continuity of value functions via Theorem 2 of \cite{milgrom2002envelope}.}

Proposition \ref{th: continuous} shows differentiability generally holds when a single demand parameter varies. In Corollary \ref{cor:multi-demand-abscont}, we extend this result to demand paths, ensuring that derivatives generally exist under vector-valued demand perturbations.

\begin{corollary}[{\cgreen Finite nondifferentiability} along demand paths]\label{cor:multi-demand-abscont}
Let $\vD \in\mathbb{R}^{|N||T|}$ be a baseline demand vector and let $\Delta \vD\in\mathbb{R}^{|N||T|}$
be any perturbation direction. For $\theta\in\mathbb{R}$, define the affine demand path $D(\theta):=\vD +\theta\,\Delta \vD$. Further, let $[\underline\theta,\overline\theta]$ be any compact interval such that the strengthened SDP relaxation $\uc^\sdpbd(D(\theta))$ is feasible for every $\theta\in[\underline\theta,\overline\theta]$. Then the scalar value function $V^{\mathrm{UC\text{-}SDP}}\bigl(D(\theta)\bigr)$ is differentiable everywhere on $[\underline\theta,\overline\theta]$ except possibly at finitely many points.
\end{corollary} 

\section{Pricing for ACUC}\label{ch: ac}
    We now extend SDP-based pricing to the ACUC problem, an area of growing interest among researchers and practitioners. In the ACUC problem, DC flow balance and limit constraints in $\uc$ \eqref{eq: uc_model_sf_1}-\eqref{eq: uc_model_sf_3} are replaced with their AC counterpart \citep{zohrizadeh2018sequential}:
    \begin{subequations}\label{eq: ac_flow}
    \begin{alignat}{4}
        &\sum_{g\in\mc{G}_k} s_{gt} - \sum_{(k, j)\in \mc{L}\cup\mc{L}^R} s_{kjt}= D_{kt} &~~&\forall k\in \mc{N}, t\in\mc{T} && (\lambda_{kt}) &&\label{eq: ac_flow_1}\\
       & s_{ijt} = v_{it}(v_{it}^c-v_{jt}^c)Y^c_{ij} &&\forall (i, j)\in\mc{L}, t\in\mc{T}\label{eq: ac_flow_2_1}\\
       & s_{jit} = v_{jt}(v_{jt}^c-v_{it}^c)Y^c_{ij} &&\forall (i, j)\in\mc{L}, t\in\mc{T}\label{eq: ac_flow_2_2}\\
       & \vmin \leq |v_{kt}|\leq \vmax &&\forall k\in \mc{N}, t\in\mc{T} \label{eq: ac_flow_2_3} \\
        & |s_{ijt}|^2 \leq (\smax)^2 && \forall (i,j)\in\mc{L}\cup\mc{L}^R, t\in\mc{T}, \label{eq: ac_flow_3}
    \end{alignat}
    \end{subequations}
    where $\mc{L}^R$ denotes the set of transmission lines $\mc{L}$ in reversed direction. For the variables, $s_{gt} = p_{gt}+\img q_{gt}$ is the complex AC power generation, where $p_{gt}$ and $q_{gt}$ respectively denote the real and reactive electricity production of generator $g$ during time period $t$. Similarly, $s_{ijt} = p_{ijt}+\img q_{ijt}$ is the complex AC power flow from bus $i$ to bus $j$ at $t$, and $v_{it}=\vr_{it}+\img \vi_{it}$ is the complex AC voltage at bus $i$ and time $t$. The complex dual variable $\lambda_{kt}:= \lambdar_{kt} + \img \lambdai_{kt}$. Constraints \eqref{eq: ac_flow_1} are the network flow constraints, which ensure that the total demand equals the total production minus total outflow at each bus and time period. The coefficient $D_{kt} = \Dr_{kt}+\img \Di_{kt}$ is the complex AC power demand. Constraints \eqref{eq: ac_flow_2_1} and \eqref{eq: ac_flow_2_2} model the power flow on each line, where $Y_{ij} = G_{ij}+\img B_{ij}$ is the complex admittance. Constraints \eqref{eq: ac_flow_2_3} and \eqref{eq: ac_flow_3} respectively model voltage limit on each bus and thermal limit on each line. 

    In addition, constraint \eqref{eq: reactive} is needed to bound reactive power output:
    \begin{align}\label{eq: reactive}
    \qmin z_{gt}  \le q_{gt} \leq \qmax z_{gt} &&\forall g\in\mc{G}, t\in\mc{T}.
    \end{align}

    The ACUC model is defined as follows: \[\ac := \min\left\{\sum_{g\in\mc{G}}\sum_{t\in\mc{T}} \left(C^p_g p_{gt} + C^{u}_g u_{gt} + C^{v}_g v_{gt} + C_g^z z_{gt}\right)|\eqref{eq: ac_flow}, \eqref{eq: reactive}, \eqref{eq: uc_model_sf_4}-\eqref{eq: uc_model_sf_13}\right\}.\]
    \subsection{SDP Relaxation for ACUC}\label{ch: ac_sdp}
    Problem $\ac$ contains both AC- and UC-related nonconvex constraints, resulting in a nonconvex feasible region. Our goal in this section is to obtain an SDP relaxation of $\ac$, which we later use to derive SDP-based pricing for ACUC.
    
    To convexify AC-related nonconvex nonlinear constraints \eqref{eq: ac_flow_2_1} - \eqref{eq: ac_flow_2_3}, we use the well-known SDP-based rectangular formulation in the literature \citep{taylor}. Specifically, an SDP relaxation that lifts the bilinear terms in $v_{it}$ is used \citep{taylor}. In this SDP relaxation, introduce $\bm{W}_t = \begin{bmatrix}\vWr_t &\vWc_t\\ \bm{W}^{C\top}_t &\vWi_t \end{bmatrix}$, where $\vWr_t = \vvr_t \vv^{R\top}_t$, $\vWi_t = \vvi_t \vv^{I\top}_t$, and $\vWc_t=\vvr_t \vv^{I\top}_t$. In other words, $\bm{W}_t = \vv \vv^\top$ where $\vv=\begin{bmatrix}\vv^{R} \\\vv^{I}\end{bmatrix}$ and is thus rank 1.
    Constraints \eqref{eq: ac_flow_2_1}-\eqref{eq: ac_flow_2_3} are relaxed as follows:
    \begin{subequations}\label{eq: rectangular}
    \begin{align}
    & p_{ijt} = G_{ij} (\Wr_{iit} + \Wi_{iit})-G_{ij}(\Wr_{ijt} + \Wi_{ijt}) - B_{ij}(\Wc_{jit} - \Wc_{ijt}) && \forall (i, j)\in\mc{L}, t\in\mc{T}\label{eq: rectangular_1}\\
    & q_{ijt} = - B_{ij}(\Wr_{iit} + \Wi_{iit}) + B_{ij}(\Wr_{ijt} + \Wi_{ijt}) - G_{ij}(\Wc_{jit} - \Wc_{ijt}) &&\forall (i, j)\in\mc{L}, t\in\mc{T}\label{eq: rectangular_2}\\\
    & p_{jit} = G_{ij} (\Wr_{jjt} + \Wi_{jjt})-G_{ij}(\Wr_{ijt} + \Wi_{ijt}) - B_{ij}(\Wc_{ijt} - \Wc_{jit}) &&\forall (i, j)\in\mc{L}, t\in\mc{T}\label{eq: rectangular_3}\\
    & q_{jit} = -B_{ij}(\Wr_{jjt} + \Wi_{jjt}) + B_{ij}(\Wr_{ijt} + \Wi_{ijt}) - G_{ij}(\Wc_{ijt} - \Wc_{jit})&& \forall (i, j)\in\mc{L}, t\in\mc{T}\label{eq: rectangular_4}\\
    & (\vmin)^2\leq \Wr_{kkt}+\Wi_{kkt}\leq (\vmax)^2 && \forall k\in\mc{N}, t\in\mc{T}\label{eq: rectangular_5}\\
    & \bm{W}_t\in \S^+_{2|\mc{N}|} && \forall t\in\mc{T}, \label{eq: rectangular_6}
    \end{align}
    \end{subequations}
    where constraints \eqref{eq: rectangular_1} and \eqref{eq: rectangular_2} are relaxations of \eqref{eq: ac_flow_2_1}. Constraints \eqref{eq: rectangular_3} and \eqref{eq: rectangular_4} are relaxations of \eqref{eq: ac_flow_2_2}. Constraints \eqref{eq: rectangular_5} are relaxations of \eqref{eq: ac_flow_2_3}. Finally, the conic constraints \eqref{eq: rectangular_6} relaxes the constraints $\bm{W}_t = \vv \vv^\top, t\in\T$. 

    With the relaxation of \eqref{eq: ac_flow_2_1} - \eqref{eq: ac_flow_2_3}, the only nonconvexity remaining in $\ac$ is driven by binary variables. We employ the same SDP relaxation technique as for DCUC on the UC-related mixed-integer linear constraints in $\ac$, including \eqref{eq: ac_flow_1}, \eqref{eq: uc_model_sf_4}-\eqref{eq: uc_model_sf_13}, and \eqref{eq: reactive}. To reduce the dimensionality of the lifted matrix, we aim to avoid introducing AC power-flow variables, such as $s_{kjt}$, in the lifting. Thus, we replace \eqref{eq: ac_flow_1} with the following aggregated demand constraints for real power when deriving the corresponding RLT constraints:
    \begin{align}\label{eq: agg_demand}
        \sum_{g\in\mc{G}} p_{gt} - \ploss_t = \sum_{k\in\mc{N}}D^R_{kt} &&\forall t\in\mc{T},
    \end{align}
    where $\ploss_t$ is the real total line loss at $t$, satisfying:
    \begin{align}\label{eq: loss}
        \ploss_t = \sum_{k\in\mc{N}}\sum_{(k, j)\in \mc{L}\cup\mc{L}^R} p_{kjt} &&\forall t\in\mc{T}.
    \end{align}

    For reactive power, the corresponding RLT constraints of aggregated demand constraints can be similarly derived. In our experiments, we observe that without such constraints, the SDP relaxation for $\ac$ is already strong compared with the alternative, and thus we do not include them.

    To obtain an SDP relaxation for the UC-related constraints \eqref{eq: loss}, \eqref{eq: reactive}, and \eqref{eq: uc_model_sf_4}-\eqref{eq: uc_model_sf_13}, we first convert inequality constraints to equality constraints, and then reformulate them with the lifting \eqref{eq: sdp_2}-\eqref{eq: sdp_6}. Then, our SDP relaxation for ACUC, denoted as $\ac^\sdp$, is as follows:
    
    \begin{subequations}\label{eq: ac_sdp}
    \begin{align}
    {\cgreen \ac^\sdp := }\min \quad& \sum_{g\in\mc{G}}\sum_{t\in\mc{T}} \left(C^p_g p_{gt} + C^{u}_g u_{gt} + C^{v}_g v_{gt} + C_g^z z_{gt}\right)\hspace{-3cm}\label{eq: ac_sdp_0}\\
    \st \quad & \eqref{eq: ac_flow_3}, \eqref{eq: rectangular}, \eqref{eq: loss} \label{eq: ac_sdp_1}\\
    & \sum_{g\in\mc{G}_k} s_{gt} - \sum_{(k, j)\in \mc{L}\cup\mc{L}^R} s_{kjt}= D_{kt} &&\forall k\in \mc{N}, t\in\mc{T}\label{eq: ac_sdp_2}\\
    &\aphitop_{egt} \vx = b_{egt} && \forall e= 1,...,m; g\in\mc{G}, t\in\mc{T}\label{eq: ac_sdp_3}\\
    &\tr(\vh^{\rm R}_t \vh^{\rm R \top}_t \bm{X}) = \left(\sum_{k\in\mc{N}}D^{\rm R}_{kt}\right)^2 &&\forall t\in\mc{T}\label{eq: ac_sdp_4}\\
    &\tr(\aphi_{egt} \aphitop_{egt} \bm{X}) = b_{egt}^2&&\forall e= 1,...,m; g\in\mc{G}, t\in\mc{T}\label{eq: ac_sdp_5}\\
    &x_i=X_{ii} && i\in\mc{B}\label{eq: ac_sdp_6}\\
    & X_{ii}\leq U_i^2 &&i\in [n]\sm\mc{B}\label{eq: ac_sdp_6_2}\\
    &\bm{Y}\in\mc{S}^+_{n+1}, \bm{Y}\geq 0,\label{eq: ac_sdp_7}
    \end{align}
    \end{subequations}
    where $\vx$ denotes the vector of all variables in UC-related constraints, including slack variables and $\ploss$. $\bm{Y}: = \begin{bmatrix} 1 &\vx^\top\\ \vx & \bm{X}\end{bmatrix}$. Constraints \eqref{eq: ac_sdp_1} are convexified AC-related constraints and definition for $\ploss$. Constraints \eqref{eq: ac_sdp_2} are flow balance constraints. Constraints \eqref{eq: ac_sdp_4} are RLT constraints obtained by squaring \eqref{eq: agg_demand}, where $\vh^{\rm R}_t$ is the vector of coefficients for real power variables in \eqref{eq: agg_demand}, i.e., entries of $\vh^{\rm R}_t$ equal 1 for $p_{gt},g\in\mc{G}$, equal $-1$ for $\ploss_t$, and 0 otherwise.
    Constraints \eqref{eq: ac_sdp_6_2} impose upper bounds on diagonal terms of $X_{ii}$ that correspond to continuous variables in $\vx$. The remaining constraints have the same definitions as their counterparts in $\uc^\sdp$, with one caveat that the operational constraints now include an equality version of \eqref{eq: reactive}.
    \subsection{Envelope Theorem for ACUC Pricing}\label{ch: ac_price}
    We now extend the results for DCUC pricing in Section \ref{ch: pricing_sdp} to ACUC. First, we show that strong duality holds for $\ac^\sdp$, as its feasible region is bounded and consists of two cones and one hyperplane:

    \begin{proposition}[Strong duality for $\ac^\sdp$]\label{th: acuc_strong_dual} Suppose that $\ac$ is feasible. Then, strong duality holds between $\ac^\sdp$ and its dual.
    \end{proposition}

    With strong duality, we can immediately extend the envelope theorem for $\mc{P}^\sdpbd$ to $\ac^\sdp$, as the proof of Theorem \ref{prop:envelopetheorem1} and Corollary \ref{cor:vector-envelope} are generally applicable for SDPs with strong duality and parameters appear only in the right-hand sides. 
    
    
    Consider a feasible ACUC SDP relaxation parameterized by demand $D_{kt}$. We obtain the partial derivative of its dual problem's objective function at an optimal solution with respect to $D^R_{kt}$ and $D^I_{kt}$, denoted as $r^\dr_{kt} := f^{ac-d}_{D^R_{kt}}(\Psi^*,D_{kt})$ and $r^\di_{kt} := f^{ac-d}_{D^I_{kt}}(\Psi^*,D_{kt})$ respectively, where $f^{ac-d}(\Psi,D_{kt})$ is the dual objective function and $\Psi^*$ is an optimal solution of the dual of $\ac^\sdp$. We use $r^d_{kt} := r^\dr_{kt}+\img r^\di_{kt}$ as the electricity price for ACUC at node $k$ and time $t$, where $r^\dr_{kt}$ and $r^\di_{kt}$ are respectively the price for real and reactive power. Proposition \ref{th: acuc_price} provides the expression for $r^d_{kt}$.

    \begin{proposition}[Expression for ACUC pricing]\label{th: acuc_price}
    Let $\lambda_{kt}=\lambda^R_{kt}+\img \lambda^I_{kt}$ be the complex dual variable associated with \eqref{eq: ac_sdp_2}, and let $\hat\lambda^R_t$ be the dual variable associated with \eqref{eq: ac_sdp_4}. The price
    \[
    r^{d\text{-}R}_{kt}=\lambda^{R*}_{kt}+2\hat\lambda^{R*}_t\sum_{k'\in N}D^R_{k't},
    \qquad
    r^{d\text{-}I}_{kt}=\lambda^{I*}_{kt}.
    \]
        \end{proposition}


    In the SDP-based mechanism for ACUC, we use the prices obtained from Proposition \ref{th: acuc_price}, which are imposed on both real and reactive power. Computing the dispatch decisions is challenging, because $\ac$ is an MINLP and can become intractable quickly as the instance scales up. In our experiments, even small instances with 14 buses are difficult to solve. We therefore use high-quality feasible dispatch solutions in the mechanism, which can be obtained via local MINLP solvers such as KNITRO \citep{byrd2006k}. The SDP-based mechanism for ACUC is defined as follows:
    \begin{definition}[SDP-based mechanism for ACUC]
        Under the SDP-based mechanism for ACUC, the system operator follows a feasible dispatch decision of $\ac$. Prices are obtained via solving the SDP relaxation $\ac^\sdp$. At time $t$, the system operator 
        \BI
        \I pays $\pi^G_{gt} = r^{\dr}_{k(g),t} p^\feas_{gt} + r^{\di}_{k(g),t} q^\feas_{gt}$ to generator $g$ at each time $t\in\T$, where $k(g)$ denotes the node at which $g$ is located, and $p^\feas_{gt}$ and $q^\feas_{gt}$ are respectively feasible real and reactive production levels from the MINLP $\ac$. In addition, the system operator pays an LOC uplift payment $U_g(\vrd)$ to generator $g$; and
        \I collects $\pi^L_t = \sum_{k\in\mc{N}}\left(r^{\dr}_{kt} D^R_{kt} + r^{\di}_{kt} D^I_{kt}\right)$ from the load at each time $t\in\T$. In addition, the system operator collects a price adder of $\frac{\sum_{g\in\G}U_g(\vrd)}{\sum_{k\in\mc{N}, t\in\T}(D^R_{kt}+D^I_{kt})}$ per unit of demand.  
        \EI 
    \end{definition}
    As in our pricing scheme for DCUC, the pricing scheme here is cost-recovering by construction. 
    
    Finally, we also extend the {\cgreen finite nondifferentiability} result to the value function of $\ac^\sdp$, denoted as $V^{\rm AC-SDP}(\Ddotkt)$, where $\Ddotkt$ a single demand parameter (either $D^R_{kt}$ or $D^I_{kt}$). The proof follows the argument of Proposition \ref{th: continuous} for $\uc^\sdp$. The only difference is that the feasible region of $\ac^\sdp$ includes additional linear matrix inequality (LMI) constraints, such as \eqref{eq: rectangular_6}. Since each spectrahedron defined by an LMI is semialgebraic and semialgebraic sets are closed under finite intersections, adding \eqref{eq: rectangular_6} preserves semialgebraicity of the relevant sets. Hence, the same conclusion holds, and we omit the proof.

    \begin{proposition}[{\cgreen Finite nondifferentiability} of $V^{\rm AC-SDP}(\Ddotkt)$]\label{th: ac_continuous}
    For any interval $[\underline{D}^{\cdot}_{kt}, \overline{D}^\cdot_{kt}]$ such that $\ac^\sdp$ is feasible for all $\Ddotkt \in [\underline{D}^{\cdot}_{kt}, \overline{D}^\cdot_{kt}]$, the value function $V^{\mathrm{AC\text{-}SDP}}$ is differentiable everywhere on $[\underline{D}^{\cdot}_{kt}, \overline{D}^\cdot_{kt}]$ except possibly at finitely many points.
    \end{proposition}

    \section{Bounding Lost Opportunity Cost}\label{ch: loc}
    In this section, we derive upper bounds on the total LOC in the SDP-based mechanism. The results are applicable to both DCUC and ACUC models. Consider a general UC problem parameterized by a demand vector $\vD\in\R^{|\mc{I}|}$, where $\mc{I}$ indexes demand at all buses in the set $\mc{N}$ and all time periods in the set $\T$. Denote each generator $g\in\G$'s decision vector as $\vx_g\in \X_g$, which includes binary commitment variables and continuous production variables, and denote the cost function as $C_g(\vx_g)$. The vector of production level variables $\vs_g$ is a subvector of $\vx_g$. Let $\vr\in\R^{\mc{I}}$ be a vector of net outflow and $\mc{R}$ be a set consisting of network feasibility constraints, which include either DC power flow constraints or AC power flow constraints. The UC problem is: 
    \begin{subequations}\label{eq: min_cost}
    \begin{alignat}{4}
        V(\vD) := &\min_{\substack{\vx_g\in \X_g, \forall g\in\G;\\ \vr\in\mc{R}}} ~~&&\sum_{g\in\G} C_g(\vx_g)\\
        &\st && \bm{A}^s \vs - \vr = \vD, \label{eq: min_cost_1}
    \end{alignat}
    \end{subequations}
    where the objective minimizes the total costs. Constraint \eqref{eq: min_cost_1} enforces flow balance, where $\bm{A}^s$ is the coefficient matrix associated with production decisions, and $\vs$ is the vector of all production level variables aggregated across generators, i.e., \(\vs = (\vs_g)_{g \in \G}\). Thus, \(\bm{A}^s \vs\) represents the total generation at each bus.
    
    We bound the total LOC with the optimality gap of a relaxation. Define the LOC of $g\in\G$ as a function of price vector $\vpi\in\R^{|\mc{I}|}$. Given $\vpi$, LOC is defined as the difference between a generator's maximum possible profit and its realized profit under the feasible decision vectors \(\vxfeas_g\) and \(\vsfeas_g\). Formally, the LOC $U_g(\vpi)$ is defined as follows:

    \begin{align}
        U_g(\vpi) := \max_{\vx_g\in\X_g}\left(\pi_{k(g)}^\top \vs_g - C_g(\vx_g)\right) - \left(\pi_{k(g)}^\top \vsfeas_g - C_g(\vxfeas_g)\right)
    \end{align}
    where \(\vpi_{k(g)}\) denotes the price vector at bus $k(g)$, the location of generator $g$. We assume that the maximum possible profit is attainable and thus $\max_{\vx_g\in\X_g}\left(\pi_{k(g)}^\top \vs_g - C_g(\vx_g)\right) = \sup_{\vx_g\in\X_g}\left(\pi_{k(g)}^\top \vs_g - C_g(\vx_g)\right)$. This assumption is without loss of generality, as generators have bounded capacities.
    
    First, we prove Lemma \ref{th: conjugate} to connect the maximum possible profit with the convex conjugate of the value function $V(\vD)$, which is needed for proving the theorem on the LOC bound. The convex conjugate of $V(\vD)$ is defined as $V^c(\vpi) := \sup_{\vD\in\R^{|\mc{I}|}} \{\vpi^\top \vD - V(\vD)\}$, where $V(\vD)=+\infty$ if \eqref{eq: min_cost} is infeasible at $\vD$. Note that under this convention, the supremum can equivalently be taken over $\rm{\bf dom} V :=\{\vD:V(\vD)<+\infty\}$, since \(\vpi^\top\vD-V(\vD)=-\infty\) whenever \(V(\vD)=+\infty\).
    
    \begin{lemma}\label{th: conjugate}
    The convex conjugate $V^c(\vpi) = \sum_{g\in\G} \max_{\vx_g\in\X_g}\left(\pi_{k(g)}^\top \vs_g - C_g(\vx_g)\right) + \sup_{\vr\in\mc{R}} (-\vpi^\top \vr)$ for $\vpi\in\R^{|\mc{I}|}$.
    \end{lemma}

    The proof relies on the fact that the feasible regions \(\X_g, \forall g\in\G\) and $\mc{R}$ are separable. While this is true for the ACUC model $\ac$, it does not generally hold for the DCUC model $\uc$ in its current form, as the flow limit constraints \eqref{eq: uc_model_sf_2} and \eqref{eq: uc_model_sf_3} depend on $p_{gt}$'s. On the other hand, the network-flow formulation of DCUC \citep{bienstock2024risk} matches the structure of \eqref{eq: min_cost}, with a feasible region that can be separated by $\X_g$'s and $\mc{R}$. Thus for DCUC, the proofs in this section rely on the network-flow formulation. Yet since the two DCUC formulations are equivalent, with identical optimal values and optimal dispatch decisions, the results on LOC bound in Theorem \ref{th: loc} and Corollary \ref{th: loc_eq} apply to DCUC regardless of the chosen formulation.

    Let $\zfeas(\vD)$ be the total costs corresponding to a feasible solution $(\vxfeas, \vrfeas)$ of \eqref{eq: min_cost} at $\vD$. Denote $\vrel(\vD)$ the value function of a relaxation of \eqref{eq: min_cost}, with $\vrel(\vD)\leq V(\vD), \vD\in\R^{|\mc{I}|}$. In Theorem \ref{th: loc}, we show that if the price $\vpi$ is a subgradient of $\vrel(\vD)$, then the LOC under $\vpi$ is bounded by the gap between $\zfeas(\vD)$ and $\vrel(\vD)$. {\cgreen The proof relies on Lemma \ref{th: conjugate}, as well as convex conjugate properties including order-reversing and the Fenchel's inequality.}

    \begin{theorem}\label{th: loc}
        Let $\vpi\in\partial \vrel(\vD)$. We have $\sum_{g\in\G}U_g(\vpi)\leq \zfeas(\vD) - \vrel(\vD)$.
    \end{theorem}
    \proof{Proof of Theorem \ref{th: loc}}
    We use the following well-known properties of convex conjugate in this proof \citep[e.g.][Section 3.3]{borwein2006convex}:
    \BI
    \I {\it Order reversing: }Denote $f\leq g$ when $f(\vx)\leq g(\vx)$ holds pointwise for all $\vx$. If $f\leq g$ then $f^c\geq g^c$.
    \I {\it Fenchel's inequality: }$\vpi^\top \vx\leq f(\vx) + f^c(\vpi)$. The equality holds only when $\vpi\in\partial f(\vx)$.
    \EI

    We have:
    \begin{align*}
        \sum_{g\in\G}\max_{\vx_g\in\X_g}\left(\pi_{k(g)}^\top \vs_g - C_g(\vx_g)\right) = & V^c(\vpi) - \sup_{\vr\in\mc{R}} \{-\vpi^\top \vr\}\\
        \leq & (\vrel)^c(\vpi) - \sup_{\vr\in\mc{R}} \{-\vpi^\top \vr\}\\
        = & \vpi^\top \vD - \vrel(\vD) - \sup_{\vr\in\mc{R}} \{-\vpi^\top \vr\},
    \end{align*}
    where the first equality follows from Lemma \ref{th: conjugate}. The inequality holds because $\vrel \leq V$ and the convex conjugate is order-reversing. The final equality follows from Fenchel's inequality, using the fact that $\vpi \in \partial \vrel(\vD)$. 

    Therefore, the LOC
    \begin{subequations}\label{eq: loc_proof}
    \begin{align}
        \sum_{g\in\G} U_g(\vpi) = &\sum_{g\in\G} \max_{\vx_g\in\X_g}\left(\pi_{k(g)}^\top \vs_g - C_g(\vx_g)\right) - \sum_{g\in\G} \left(\pi_{k(g)}^\top \vsfeas_g - C_g(\vxfeas_g)\right)\\
        \leq & \left(\vpi^\top \vD - \vrel(\vD) - \sup_{\vr\in\mc{R}} \{-\vpi^\top \vr\}\right) - \left(\sum_{g\in\G} \pi_{k(g)}^\top \vsfeas_g - \sum_{g\in\G} C_g(\vxfeas_g)\right)\\
        = & \left(\vpi^\top \vD - \vrel(\vD) - \sup_{\vr\in\mc{R}} \{-\vpi^\top \vr\}\right) - \left(\sum_{g\in\G} \pi_{k(g)}^\top \vsfeas_g - \zfeas(\vD)\right)\\
        = & \left(\vpi^\top \vD - \vrel(\vD) - \sup_{\vr\in\mc{R}} \{-\vpi^\top \vr\}\right) - \left(\vpi^\top (D+\vrfeas) - \zfeas(\vD)\right)\\
        = & \zfeas(\vD) - \vrel(\vD) - \left(\sup_{\vr\in\mc{R}}\{-\vpi^\top \vr\} - (-\vpi^\top \vrfeas)\right)\\
        \leq & \zfeas(\vD) - \vrel(\vD),
    \end{align}
    \end{subequations}
    where the third equality follows from $\sum_{g\in\G} \pi_{k(g)}^\top \vsfeas_g = \vpi^\top \bm{A}^s \vsfeas = \vpi^\top(D + \vrfeas)$.
    \qed
    \endproof

    Note that while our work focuses on SDP-based pricing, Theorem \ref{th: loc} is more general and applies to any relaxation of UC that satisfies $\vrel(\vD)\leq V(\vD)$ at any $\vD\in\R^{|\mc{I}|}$.

    For our SDP-based price vector, a sufficient condition for it to be a subgradient of the SDP relaxation's value function is for $\vD$ to enter linearly in the right-hand sides of the SDP. This is because for a parameterized SDP where the parameter appears only linearly in the right-hand sides, its value function is convex \citep[Section 5.6.1]{boyd2004convex} and thus its gradients are the same as subgradients. Note that the right-hand side parameters of $\uc^\sdp$ or $\ac^\sdp$ are not generally linear in $\vD$. Yet in numerical experiments (Section \ref{ch: loc_result}), our proposed SDP-based pricing leads to lower LOCs compared with the fixed-binary pricing, which are widely-used in practice.
    
    One method to ensure linearity of SDP relaxations' right hand sides is to replace RLT constraints that introduce quadratic terms in $\vD$. For example, in $\ac^\sdp$, we can replace \eqref{eq: ac_sdp_2} with valid inequalities where $\vD$ enters linearly in the right-hand sides, such as $\tr(\vh^{\rm R}_t \vh^{\rm R \top}_t \bm{X}) \leq M\sum_{k\in\mc{N}}D^{\rm R}_{kt}$, with $M> \sum_{k\in\mc{N}}D^{\rm R}_{kt}$ being a large constant. Deriving such valid inequalities is a topic of future work. 

    For the bound in Theorem \ref{th: loc} to be tight, a simple case is when $\vrel(\vD)$ is exact and thus $\vrel(\vD) = V(\vD)$. In this case, the subgradient from $\vrel(\vD)$ yields a price vector that entails no loss of opportunity cost. 

    A more general case for the bound to be tight with a nonzero lost opportunity cost arises when there is no congestion in the network and when the convex conjugates of $V(\vpi)$ and $\vrel(\vpi)$ coincide, {\cgreen as these conditions eliminate the two inequalities in the proof of Theorem \ref{th: loc_eq}:} 

    \begin{corollary}\label{th: loc_eq}
        Let $\vpi\in\partial \vrel(\vD)$. If (1) there is no congestion in the network and (2) $V^c(\vpi)=(\vrel)^c(\vpi)$, then $\sum_{g\in\G}U_g(\vpi) = \zfeas(\vD) - \vrel(\vD)$.
    \end{corollary}

    An important example in which the conditions of Corollary \ref{th: loc_eq} hold is convex-hull pricing for DCUC without network constraints, since $\zfeas(\vD) = V(\vD)$ corresponds to an optimal solution. In this setting, \citet{gribik2007market} show that the convex hull prices, which are derived as Lagrange multipliers $\vpi^{\rm L}$ from a Lagrangian relaxation, ensure $V^{\rm L}(\vD) = V^{cc}(\vD)$, where $V^{\rm L}(\vD)$ and $V^{cc}(\vD)$ are respectively the Lagrangian relaxation and biconjugate of $V(\vD)$. As a result, $V^c(\vpi^{\rm L}) = (V^{cc})^c(\vpi^{\rm L}) = (V^{L})^c(\vpi^{\rm L})$, which coincides with the second condition in Corollary \ref{th: loc_eq}. Thus, under convex-hull pricing and when there is no congestion, the bound of Theorem \ref{th: loc} is tight. 
    \section{Numerical Experiments}\label{sec:pricingscheme}
     In this section, we present multiple case studies for DCUC and ACUC. Our experiments are implemented in Julia 1.12.4 using the optimization modeling package JuMP 1.25.0 \citep{legat2022mathoptinterface}. Optimization problems are solved on the Palmetto Cluster \citep{antao2024modernizing} with 10 Intel Xeon CPUs and 350GB of memory. We use Gurobi 12.0.1 to solve MILP and linear programming (LP) models, while MOSEK 11.0.0 \citep{aps2020mosek} is used for SDP models. 
     
     We use the data from the UnitCommitment.jl package \citep{alinson_s_xavier_2022_6857290} for UC coefficients. Specifically, we use the MATPOWER/UW-PSTCA instances \citep{zimmerman2010matpower, IEEE} corresponding to January 1st from the package. In addition, the data from the PowerModels.jl package \citep{8442948} are used for AC power flow coefficients in the ACUC model. To generate multiple scenarios for each instance, we multiply loads at all buses by a load multiplier chosen from a set of 13 discrete values $\{0.1,...,1.3\}$.

    We impose several algorithmic enhancements \citep{fattahi2017conic} to improve the tightness or tractability of the SDP relaxation for both DCUC and ACUC. To improve tightness, we impose triangle inequalities on some binary variables (see Section \ref{ch: ineq} of the Appendix for detail). To improve tractability, we relax the PSD conic constraint $\bm{Y}\in\mc{S}^+_{n+1}$ with a block-diagonal relaxation, decomposing it by time period. Specifically, for each time period $t\in\mc{T}$ we impose a PSD conic constraint on a submatrix of $\bm{Y}$, consisting of all entries of $\vx$ and $\bm{X}$ corresponding to $t$. This relaxation enables us to solve larger instances while maintaining a tight formulation in our experiments.


\subsection{DCUC Case Study} \label{ch: dcuc_results}
In this section, we assess the tightness of $\uc^\sdpbd$ (with algorithmic enhancements) by presenting its optimality gap relative to the optimal value of $\uc$. We also compare this gap with that of the linear relaxation $\mc{UC}^{\text{LP}}$, obtained by relaxing all binary variables to continuous variables in $[0,1]$. The optimality gap is calculated as follows: $$\text{Optimality gap of }\mc{UC}^{\rel}= \frac{\text{opt}(\mc{UC})-\text{opt}(\mc{UC}^{\rel})}{\text{opt}(\mc{UC})},$$
where $\text{opt}(\cdot)$ denotes the optimal value of a problem. $\mc{UC}^{\rel}$ denotes either $\uc^\sdpbd$ or $\uc^{\text{LP}}$.

Table \ref{table: dcuc_mult_time} summarizes the optimality gaps of $\mc{UC}^{\text{LP}}$ and $\mc{UC}^\sdpbd$ for the IEEE 14-, 30-, and 57-bus instances for each load multiplier. The 14-bus and 30-bus instances were solved over 24 time periods, whereas the 57-bus instance was solved over 6 time periods, representing the largest instance size we solve to optimality. 

With the default tolerance setting of MOSEK, some instances have negative optimality gaps of $-0.1\%$. These are possibly due to numerical issues, as when we use a tighter interior-point optimizer tolerance for primal feasibility and relative gap termination (tightened from the default $e^{-8}$ to $e^{-10}$), such negative gaps disappear. We thus report results for these instances under the tighter tolerance, as noted by $*$ in the table.


\begin{table}[htbp!]
\centering
         \caption{Optimality gap (\%) comparisons for DCUC (*: tighter solver tolerance used.)}\label{table: dcuc_mult_time}
{\footnotesize
\begin{tabular}{rrrrrrr}
\toprule
\multicolumn{1}{c}{Load} & \multicolumn{2}{c}{ 14-bus} &
 \multicolumn{2}{c}{ 30-bus}  & \multicolumn{2}{c}{ 57-bus}\\

\multicolumn{1}{c}{multiplier} & \multicolumn{1}{c}{$\mc{UC}^{\text{LP}}$}           & \multicolumn{1}{c}{$\mc{UC}^\sdpbd$} & \multicolumn{1}{c}{$\mc{UC}^{\text{LP}}$} & \multicolumn{1}{r}{$\mc{UC}^\sdpbd$}  & \multicolumn{1}{c}{$\mc{UC}^{\text{LP}}$} & \multicolumn{1}{r}{$\mc{UC}^\sdpbd$} \\ 
\cmidrule(l){1-1} \cmidrule(l){2-3} \cmidrule(l){4-5} \cmidrule(l){6-7}   \multicolumn{1}{r}{0.1} &  \multicolumn{1}{r}{33.6} &  \multicolumn{1}{r}{0.0*} &  \multicolumn{1}{r}{28.8} &  \multicolumn{1}{r}{0.0*}&  \multicolumn{1}{r}{10.9} &  \multicolumn{1}{r}{0.0*}\\ 
\multicolumn{1}{r}{0.2} &  \multicolumn{1}{r}{13.2} &  \multicolumn{1}{r}{0.0*}  &  \multicolumn{1}{r}{9.2} &  \multicolumn{1}{r}{0.0*}&  \multicolumn{1}{r}{3.1} &  \multicolumn{1}{r}{1.5}\\ 

\multicolumn{1}{r}{0.3} &  \multicolumn{1}{r}{3.3} &  \multicolumn{1}{r}{0.0*} &  \multicolumn{1}{r}{10.3} &  \multicolumn{1}{r}{8.4}&  \multicolumn{1}{r}{2.8} &  \multicolumn{1}{r}{1.5}\\ 

\multicolumn{1}{r}{0.4} &  \multicolumn{1}{r}{7.4} &  \multicolumn{1}{r}{5.5} &  \multicolumn{1}{r}{9.9} &  \multicolumn{1}{r}{8.4}&  \multicolumn{1}{r}{2.1} &  \multicolumn{1}{r}{1.8}\\ 

\multicolumn{1}{r}{0.5} &  \multicolumn{1}{r}{8.6} &  \multicolumn{1}{r}{7.3} &  \multicolumn{1}{r}{4.6} &  \multicolumn{1}{r}{3.7}&  \multicolumn{1}{r}{2.3} &  \multicolumn{1}{r}{1.4}\\ 

\multicolumn{1}{r}{0.6} &  \multicolumn{1}{r}{3.9} &  \multicolumn{1}{r}{3.1} &  \multicolumn{1}{r}{3.4} &  \multicolumn{1}{r}{2.8}&  \multicolumn{1}{r}{4.7} &  \multicolumn{1}{r}{0.8}\\ 

\multicolumn{1}{r}{0.7} &  \multicolumn{1}{r}{3.4} &  \multicolumn{1}{r}{2.7} &  \multicolumn{1}{r}{4.5} &  \multicolumn{1}{r}{4.0}&  \multicolumn{1}{r}{1.0} &  \multicolumn{1}{r}{0.0*}\\ 

\multicolumn{1}{r}{0.8} &  \multicolumn{1}{r}{4.9} &  \multicolumn{1}{r}{4.2} &  \multicolumn{1}{r}{3.1} &  \multicolumn{1}{r}{2.6}&  \multicolumn{1}{r}{2.5} &  \multicolumn{1}{r}{0.0*}\\ 

\multicolumn{1}{r}{0.9} &  \multicolumn{1}{r}{3.4} &  \multicolumn{1}{r}{2.9} &  \multicolumn{1}{r}{2.2} &  \multicolumn{1}{r}{1.8}&  \multicolumn{1}{r}{1.5} &  \multicolumn{1}{r}{0.0*}\\ 

\multicolumn{1}{r}{1.0} &  \multicolumn{1}{r}{3.1} &  \multicolumn{1}{r}{2.7} &  \multicolumn{1}{r}{3.1} &  \multicolumn{1}{r}{2.7}&  \multicolumn{1}{r}{1.0} &  \multicolumn{1}{r}{0.0*}\\ 

\multicolumn{1}{r}{1.1} &  \multicolumn{1}{r}{1.3} &  \multicolumn{1}{r}{0.9} &  \multicolumn{1}{r}{2.7} &  \multicolumn{1}{r}{2.3}&  \multicolumn{1}{r}{0.6} &  \multicolumn{1}{r}{0.0*}\\ 

\multicolumn{1}{r}{1.2} &  \multicolumn{1}{r}{2.7} &  \multicolumn{1}{r}{2.2} &  \multicolumn{1}{r}{1.8} &  \multicolumn{1}{r}{1.1}&  \multicolumn{1}{r}{0.4} &  \multicolumn{1}{r}{0.0*}\\ 

\multicolumn{1}{r}{1.3} &  \multicolumn{1}{r}{3.0} &  \multicolumn{1}{r}{2.4} &  \multicolumn{1}{r}{3.0} &  \multicolumn{1}{r}{1.8}&  \multicolumn{1}{r}{0.2} &  \multicolumn{1}{r}{0.0*}\\ \bottomrule
\end{tabular}}
\end{table}

For all instances, the SDP relaxation consistently yields smaller optimality gaps than the LP relaxation. The average optimality gap across all instances for the LP and SDP relaxations are 5.4\% and 2.1\%, respectively. Moreover, the optimality gaps show a decreasing trend as the problem size increases. For example, when the load multiplier equals 0.8, the optimality gaps of $\uc^\sdpbd$ for the 14-, 30-, and 57-bus instances are 4.2\%, 2.6\%, and 0.0\%, respectively. 

\subsection{ACUC Case Study}
In this section, we assess the tightness of $\ac^\sdp$ (with algorithmic enhancements). It is difficult to obtain optimality gaps for ACUC instances, as the exact ACUC MINLP model $\ac$ is difficult to solve. Therefore, we instead report the objective values of solved instances. We compare with a relaxed binary relaxation $\ac^{\text{RB}}$, obtained by (i) relaxing all binary variables in $\ac$ to continuous variables in $[0,1]$ and (ii) replacing AC-related nonconvex nonlinear constraints \eqref{eq: ac_flow_2_1} - \eqref{eq: ac_flow_2_3} with the SDP-based rectangular formulation \eqref{eq: rectangular}. Therefore, $\ac^\sdp$ and $\ac^{\text{RB}}$ differ only in how the UC-related constraints are relaxed, with $\ac^\sdp$ employing the lifting \eqref{eq: sdp_2}-\eqref{eq: sdp_6}, which is strictly stronger than a direct relaxation of the binary variables.

Table \ref{table: acuc_obj_val} summarizes the objective values of $\ac^\sdp$ and $\mc{AC}^{\text{RB}}$ for the IEEE 14-, 30-, and 57-bus instances for each load multiplier. Each instance was solved over 24 time periods. Instances that terminated with a primal status ``UNKNOWN\_RESULT\_STATUS" are indicated by ``*" in the table, while those terminated with a primal status ``NO\_SOLUTION" are indicated with ``-".

\begin{table}[htbp!]
\centering
    \begin{footnotesize}
         \caption{Objective value (\$) comparisons for ACUC (*: Instances with ``UNKNOWN\_RESULT\_STATUS", \\ -: Instances with ``NO\_SOLUTION")}\label{table: acuc_obj_val}
    
\begin{tabular}{rrrrrrr}

\toprule
\multicolumn{1}{c}{Load} & \multicolumn{2}{c}{ 14-bus} &
 \multicolumn{2}{c}{ 30-bus}  & \multicolumn{2}{c}{ 57-bus}\\

\multicolumn{1}{c}{multiplier} & \multicolumn{1}{c}{$\ac^\sdp$}           & \multicolumn{1}{c}{$\mc{AC}^{\text{RB}}$}  & \multicolumn{1}{c}{$\ac^\sdp$}           & \multicolumn{1}{c}{$\mc{AC}^{\text{RB}}$}  & \multicolumn{1}{c}{$\ac^\sdp$} & \multicolumn{1}{c}{$\mc{AC}^{\text{RB}}$}   \\ 
\cmidrule(l){1-1} \cmidrule(l){2-3} \cmidrule(l){4-5} \cmidrule(l){6-7} 

\multicolumn{1}{c}{0.1} &  \multicolumn{1}{r}{ 246739.6} &  \multicolumn{1}{r}{21129.2} &  \multicolumn{1}{r}{267818.3} &  \multicolumn{1}{r}{24633.0*}&  \multicolumn{1}{r}{ 210252.7} &  \multicolumn{1}{r}{55828.3*}\\

\multicolumn{1}{c}{0.2} &  \multicolumn{1}{r}{259109.3} &  \multicolumn{1}{r}{42291.3} &  \multicolumn{1}{r}{287728.0} &  \multicolumn{1}{r}{49313.5*}&  \multicolumn{1}{r}{256692.4} &  \multicolumn{1}{r}{112766.1*}\\ 

\multicolumn{1}{c}{0.3} & \multicolumn{1}{r}{276211.1} &  \multicolumn{1}{r}{63822.1*} &  \multicolumn{1}{r}{ 307629.2} &  \multicolumn{1}{r}{74360.0*}&  \multicolumn{1}{r}{ 306047.4} &  \multicolumn{1}{r}{170734.6*}\\

\multicolumn{1}{c}{0.4} &  \multicolumn{1}{r}{293372.9} &  \multicolumn{1}{r}{85529.5*} &  \multicolumn{1}{r}{327875.2} &  \multicolumn{1}{r}{100259.0*}&  \multicolumn{1}{r}{ 458938.5} &  \multicolumn{1}{r}{229032.7*}\\ 

\multicolumn{1}{c}{0.5} &  \multicolumn{1}{r}{310988.6} &  \multicolumn{1}{r}{107317.4*} &  \multicolumn{1}{r}{348964.1} &  \multicolumn{1}{r}{126284.9*}&  \multicolumn{1}{r}{685221.8} &  \multicolumn{1}{r}{289436.6}\\ 

\multicolumn{1}{c}{0.6} &  \multicolumn{1}{r}{ 329386.2} &  \multicolumn{1}{r}{129058.9*} &  \multicolumn{1}{r}{371336.5} &  \multicolumn{1}{r}{152268.1*}&  \multicolumn{1}{r}{917233.5} &  \multicolumn{1}{r}{361122.5}\\ 

\multicolumn{1}{c}{0.7} &  \multicolumn{1}{r}{348417.5} &  \multicolumn{1}{r}{150893.9*} &  \multicolumn{1}{r}{395784.9} &  \multicolumn{1}{r}{178408.0*}&  \multicolumn{1}{r}{1155242.1} &  \multicolumn{1}{r}{496195.8}\\ 

\multicolumn{1}{c}{0.8} &  \multicolumn{1}{r}{ 369245.2*} &  \multicolumn{1}{r}{172948.4*} &  \multicolumn{1}{r}{463715.1*} &  \multicolumn{1}{r}{204506.6*}&  \multicolumn{1}{r}{-} &  \multicolumn{1}{r}{ 708813.3*}\\ 

\multicolumn{1}{c}{0.9} &  \multicolumn{1}{r}{416871.9*} &  \multicolumn{1}{r}{195332.1*} &  \multicolumn{1}{r}{758598.3*} &  \multicolumn{1}{r}{233887.2*}&  \multicolumn{1}{r}{-} &  \multicolumn{1}{r}{940014.9}\\ 

\multicolumn{1}{c}{1.0} &  \multicolumn{1}{r}{806954.2*} &  \multicolumn{1}{r}{217885.2*} &  \multicolumn{1}{r}{1406965.0*} &  \multicolumn{1}{r}{261744.7*}&  \multicolumn{1}{r}{-} &  \multicolumn{1}{r}{1174381.5}\\ 

\multicolumn{1}{c}{1.1} &  \multicolumn{1}{r}{843345.6*} &  \multicolumn{1}{r}{241215.1*} &  \multicolumn{1}{r}{-} &  \multicolumn{1}{r}{291853.1*}&  \multicolumn{1}{r}{-} &  \multicolumn{1}{r}{1423090.0}\\ 

\multicolumn{1}{c}{1.2} &  \multicolumn{1}{r}{-} &  \multicolumn{1}{r}{264320.8*} &  \multicolumn{1}{r}{-} &  \multicolumn{1}{r}{325184.9*}&  \multicolumn{1}{r}{-} &  \multicolumn{1}{r}{1641650.0}\\ 

\multicolumn{1}{c}{1.3} &  \multicolumn{1}{r}{-} &  \multicolumn{1}{r}{289141.2*} &  \multicolumn{1}{r}{-} &  \multicolumn{1}{r}{363458.7*}&  \multicolumn{1}{r}{-} &  \multicolumn{1}{r}{-}\\ \bottomrule

\end{tabular}
\end{footnotesize}
\end{table}

For all instances that terminated with a solution, our SDP relaxation obtains tighter relaxation bounds than $\ac^{\rm RB}$, often by a large margin. This highlights the importance of lifting the UC-related constraints in strengthening the convex relaxation of ACUC. In addition, our SDP relaxation produces more reliable solutions, as it reports substantially fewer instances with an unknown result status. One possible explanation is that the tighter relaxation leads to better numerical behavior.

As the load multiplier increases, instances become increasingly difficult to solve, likely due to higher network congestion. A promising future direction is to improve the scalability of AC power flow constraints, by breaking large PSD constraints into smaller blocks via chordal decomposition. This approach is helpful to reduce solver memory usage and has been shown to improve the tractability of ACOPF problems.

Finally, we briefly comment on the solvability of the MINLP $\ac$. We experimented with solving a small $\ac$ instance with 14 buses and a single time period, omitting ramping and minimum up/down time constraints \eqref{eq: uc_model_sf_8}-\eqref{eq: uc_model_sf_11}. We used Gurobi 12.0.3, which provides a global solver for MINLPs. At a load multiplier 1.0, the solver reached a 0.4\% gap in 433.58 seconds. Compared with the upper bound provided by this run, our SDP relaxation counterpart (without algorithmic enhancements) has a 2.4\% optimality gap, indicating that the SDP relaxation is relatively tight. The $\ac$ instance becomes significantly harder at a load multiplier of 0.1, with the solver returning a 12.0\% gap after 3500 seconds. The worse performance at a low load level appears to be instance-dependent, and could be induced by e.g., tight lower-bound constraints that complicate the MINLP solution process.

\subsection{LOC for DCUC}\label{ch: loc_result}
In this section, we compare the total LOC of our SDP-based pricing scheme with the fixed-binary pricing scheme that is currently used in practice. In the fixed-binary pricing for DCUC, denoted as $\mc{UC}^{\text{FB}}$, the binary variables in $\uc$ are fixed to their optimal levels, with prices obtained from the dual solution of the resulting linear program. 

We report LOC comparisons only for DCUC, as extending fixed-binary pricing to ACUC is challenging: ACUC is an MINLP that is hard to solve, and fixing binary variables does not remove nonconvexities arising from AC power flow constraints, complicating price construction.

Table \ref{table: loc_sdp_fixbin} summarizes the total LOC for the same DCUC instances solved in Section \ref{ch: dcuc_results}, under both the fixed-binary pricing scheme ($\uc^{\text{FB}}$) and our proposed SDP-based pricing scheme ($\uc^\sdpbd$). The instances where SDP-based pricing produces a lower total LOC are highlighted with boldface. For most instances, the total LOC under our SDP-based pricing scheme is lower, with an average reduction of 46\%.

\begin{table}[htbp!]
\centering
    \begin{footnotesize}
         \caption{Total LOC (\$) comparisons for DCUC ({\bf boldface}: lower total LOC with SDP-based pricing)}\label{table: loc_sdp_fixbin}
{\cgreen 
\begin{tabular}{rrrrrrr}

\toprule
\multicolumn{1}{c}{Load} & \multicolumn{2}{c}{ 14-bus} &
 \multicolumn{2}{c}{ 30-bus}  & \multicolumn{2}{c}{ 57-bus}\\

\multicolumn{1}{c}{multiplier} & \multicolumn{1}{c}{$\mc{UC}^{\text{FB}}$}           & \multicolumn{1}{c}{$\mc{UC}^\sdpbd$}  & \multicolumn{1}{c}{$\mc{UC}^{\text{FB}}$}           & \multicolumn{1}{c}{$\mc{UC}^\sdpbd$}  & \multicolumn{1}{c}{$\mc{UC}^{\text{FB}}$}           & \multicolumn{1}{c}{$\mc{UC}^\sdpbd$}   \\ 

\cmidrule(l){1-1} \cmidrule(l){2-3} \cmidrule(l){4-5} \cmidrule(l){6-7} 

\multicolumn{1}{c}{0.1} &  \multicolumn{1}{r}{14919.6} &  \multicolumn{1}{r}{15009.1} &  \multicolumn{1}{r}{14920.0} &  \multicolumn{1}{r}{\bf14642.0}&  \multicolumn{1}{r}{3872.7} &  \multicolumn{1}{r}{\bf3782.5}\\ 

\multicolumn{1}{c}{0.2} &  \multicolumn{1}{r}{14922.3} &  \multicolumn{1}{r}{\bf14905.7} &  \multicolumn{1}{r}{14923.2} &  \multicolumn{1}{r}{\bf14664.9}&  \multicolumn{1}{r}{4627.5} &  \multicolumn{1}{r}{\bf1277.9}\\ 

\multicolumn{1}{c}{0.3} &  \multicolumn{1}{r}{14925.0} &  \multicolumn{1}{r}{14939.4} &  \multicolumn{1}{r}{18710.6} &  \multicolumn{1}{r}{48437.1}&  \multicolumn{1}{r}{7873.9} &  \multicolumn{1}{r}{\bf1339.1}\\ 

\multicolumn{1}{c}{0.4} &  \multicolumn{1}{r}{18711.0} &  \multicolumn{1}{r}{\bf6935.6} &  \multicolumn{1}{r}{25243.2} &  \multicolumn{1}{r}{\bf11207.3}&  \multicolumn{1}{r}{8847.8} &  \multicolumn{1}{r}{\bf2298.9}\\

\multicolumn{1}{c}{0.5} &  \multicolumn{1}{r}{30886.9} &  \multicolumn{1}{r}{\bf10571.0} &  \multicolumn{1}{r}{30887.0} &  \multicolumn{1}{r}{\bf6837.3}&  \multicolumn{1}{r}{11400.1} &  \multicolumn{1}{r}{\bf1481.0}\\ 

\multicolumn{1}{c}{0.6} &  \multicolumn{1}{r}{30887.0} &  \multicolumn{1}{r}{\bf5852.8} &  \multicolumn{1}{r}{28530.1} &  \multicolumn{1}{r}{\bf5703.4}&  \multicolumn{1}{r}{16010.8} &  \multicolumn{1}{r}{\bf10585.8}\\ 

\multicolumn{1}{c}{0.7} &  \multicolumn{1}{r}{28530.1} &  \multicolumn{1}{r}{\bf8253.6} &  \multicolumn{1}{r}{36001.8} &  \multicolumn{1}{r}{\bf9324.9}&  \multicolumn{1}{r}{16009.5} &  \multicolumn{1}{r}{\bf10398.6}\\

\multicolumn{1}{c}{0.8} &  \multicolumn{1}{r}{36001.7} &  \multicolumn{1}{r}{\bf9128.0} &  \multicolumn{1}{r}{45878.8} &  \multicolumn{1}{r}{\bf6386.5}&  \multicolumn{1}{r}{3448.8} &  \multicolumn{1}{r}{\bf3439.0}\\

\multicolumn{1}{c}{0.9} &  \multicolumn{1}{r}{37737.0} &  \multicolumn{1}{r}{\bf8599.3} &  \multicolumn{1}{r}{41677.4} &  \multicolumn{1}{r}{\bf55359.8}&  \multicolumn{1}{r}{3450.1} &  \multicolumn{1}{r}{\bf3442.6}\\ 

\multicolumn{1}{c}{1.0} &  \multicolumn{1}{r}{41677.1} &  \multicolumn{1}{r}{\bf8218.1} &  \multicolumn{1}{r}{49149.1} &  \multicolumn{1}{r}{\bf 8481.3}&  \multicolumn{1}{r}{3451.4} &  \multicolumn{1}{r}{ 3475.0}\\ 

\multicolumn{1}{c}{1.1} &  \multicolumn{1}{r}{41677.8} &  \multicolumn{1}{r}{\bf3569.8} &  \multicolumn{1}{r}{53397.2} &  \multicolumn{1}{r}{\bf8553.6}&  \multicolumn{1}{r}{3452.6} &  \multicolumn{1}{r}{3478.2}\\ 

\multicolumn{1}{c}{1.2} &  \multicolumn{1}{r}{49727.6} &  \multicolumn{1}{r}{\bf8742.7} &  \multicolumn{1}{r}{55556.6} &  \multicolumn{1}{r}{\bf6792.0}&  \multicolumn{1}{r}{3453.9} &  \multicolumn{1}{r}{3495.1}\\ 

\multicolumn{1}{c}{1.3} &  \multicolumn{1}{r}{54932.1} &  \multicolumn{1}{r}{\bf10516.5} &  \multicolumn{1}{r}{65625.9} &  \multicolumn{1}{r}{\bf16875.0}&  \multicolumn{1}{r}{3455.2} &  \multicolumn{1}{r}{3530.3}\\ \bottomrule

\end{tabular}
}
\end{footnotesize}

\end{table}

    \section{Conclusion} \label{ch: conclusion}
    SDPs provide tight relaxations for discrete and nonlinear market-clearing problems, making them useful tools for assessing the marginal effect of demand on social welfare. Based on this perspective, we develop a tractable, SDP-based pricing scheme applicable to electricity markets modeled with both DCUC and ACUC.

    There are several promising future directions. Motivated by the increasing penetration of weather-dependent generation (e.g., solar and wind), it is useful to consider uncertainty in the market-clearing model and how it reshapes incentives. Our SDP-based framework is well-suited for such extensions, as common stochastic modeling tools, including chance constraints and distributionally robust optimization, also have conic programming representations. In addition, it would also be helpful to further improve the tractability of the SDP relaxation by exploiting its sparsity structure and by identifying strengthening valid inequalities.


    \FloatBarrier

\bibliographystyle{informs2014} 
\begingroup
    \setlength{\bibsep}{-0.5pt}
    \linespread{1}\selectfont
\bibliography{reference}
\endgroup

%
\newpage
\begin{APPENDIX}{Appendix}

\section{Proofs for Section \ref{ch: pricing_sdp}}\label{ch: proof}
    \proof{Proof of Proposition \ref{th: dual}}
    
    {\it Step 1:} We first show that {\cgreen $\mc{P}^\sdpbd(\theta)$} is feasible. By assumption, $\mc{P}^\mbqp$ is feasible and that Assumption \ref{ass: bound} holds. We fix $\theta$ and write $b_j(\theta)$ as $b_j$ {\cgreen and $\mc{P}^\sdpbd(\theta)$ as $\mc{P}^\sdpbd$} for concreteness. Let $\bm{x}$ be a feasible solution to $\mc{P}^\mbqp$. Define $\bm{X}=\bm{x}\bm{x}^\top$ and {\cgreen $\bm{Y} = \begin{bmatrix}1 & \vx^\top \\ \vx & \bm{X}\end{bmatrix}$}. Then, $\bm{Y}\succeq \bm{0}$ and $\bm{Y}\geq \bm{0}$ since $\bm{x}\geq \bm{0}$. Moreover, we have $\mathrm{tr}(\bm{a}_j \bm{a}_j ^\top \bm{X})=(\bm{a}_j^\top \bm{x})^2=b_j^2$, {\cgreen $x_i = X_{ii}$ (for binary variables), and $X_{ii}\le U_i^2$ (for continuous variables)}, and thus $(\bm{x}, \bm{X}, \bm{Y})$ is feasible in $\mc{P}^\sdpbd$. Hence, $\mc{P}^\sdpbd$ is feasible. \\
    
    \noindent{\it Step 2:} We now prove that the feasible region of $\mc{P}^\sdpbd$ is {\cgreen bounded}. By {\cgreen \eqref{eq: sdp_bd3}} we have $X_{ii} \leq U_{i}^2$ for each $i$ where $x_i$ is continuous. On the other hand, for $X_{ii}$ where $i \in \mathcal{B}$ we have $x_i=X_{ii}$. {\cgreen Since $\bm{Y}\succeq 0$, we also have $\begin{vmatrix}1 & x_i\\x_i & X_{ii}\end{vmatrix}\geq 0 \implies$} $X_{ii}\geq x_{i}^2$. {\cgreen Thus,} $x_{i}\geq x_{i}^2 \geq 0 \implies$ {\cgreen $x_i\in [0,1] \implies$} $X_{ii} \in [0, 1]$, and thus we also have the bound $X_{ii} \leq U_i:=1$ for $i \in \mathcal{B}$. Therefore, by the $2 \times 2$ minors constraints of $\bm{X}\succeq \bm{0}$, this implies that $X_{ij}^2 \leq X_{ii}X_{jj}\leq U_i^2U_j^2$, and hence $(\bm{x}, \bm{X}, \bm{Y})$ is bounded.

    {\cgreen A quick note that the feasible region is also closed and thus is compact, a fact that is useful for the proof of Proposition \ref{th: continuous}: The feasible region is closed } since it is the intersection of finitely many closed sets: {\cgreen closed convex cones and a set with finitely many linear equalities and inequalities}. Combined with the boundedness result shown above, the feasible region is thus compact.\\ 

    {\cgreen \noindent{\it Step 3: }We now prove that strong duality holds for a homogenized $\mc{P}^\sdpbd$. This step follows the same argument as the proof of Theorem 6.1 in \citep{brown2024copositive}, which proves the strong duality of a homogenized CPP reformulation of MBQP, by showing that the homogenized CPP reformulation's feasible region is nonempty, bounded, and consists of two cones and one hyperplane. Consequently, it has strong duality by Corollary 1.2 of \citep{kim2025strong}. 
    
    For our proof, the main difference is that $\mc{P}^\sdpbd$ no longer has the completely positive conic constraint, and instead has constraints \eqref{eq: sdp_bd3}, $Y\geq 0$, and $Y\in\mc{S}_{n+1}^+$, and thus feasible region of the corresponding homogenized problem consists of slightly different cones. Specifically, consider the following homogenized SDP relaxation of $\mc{P}^\mbqp$:
    \begin{subequations}\label{eq: hom_sdp}
    \begin{align}
    \mc{P}^\hom: \min~~&\tr(\bm{C} \Ytilde)\\
    \st ~~& \tr(\Atilde_j \Ytilde) = 0 && \forall j = 1,...,m && (\sigmatilde_j) \label{eq: hom_sdp_1}\\
    & \tr(\bm{B}_j \Ytilde) = 0 && \forall j\in\mc{B} && (\kappatilde_j)\label{eq: hom_sdp_2}\\
    & \tr(\Gtilde_j \Ytilde) \le 0 && \forall j\in [n]\sm\mc{B} && (\phitilde_j)\label{eq: hom_sdp_3}\\
    & \tr(\bm{H} \Ytilde) = 1 && && (\mutilde) \label{eq: hom_sdp_4}\\
    & \Ytilde \in \mc{S}^{n+1}_+ && && (\Omegatilde)\label{eq: hom_sdp_5}\\
    & \Ytilde \ge \vo && && (\Thetatilde), \label{eq: hom_sdp_6}
    \end{align}
    \end{subequations}
    where 
    \[
    \bm{C} = \begin{bmatrix} 0 & \vc^\top \\ \vc &  \bm{Q} \end{bmatrix}, \Atilde_j = \begin{bmatrix} b_j^2 & - b_j \va_j^\top \\ - b_j \va_j & \va_j \va_j^\top \end{bmatrix}, \bm{B}_j = \begin{bmatrix} 0 & \frac{1}{2}\ve_j^\top \\ \frac{1}{2}\ve_j & -\ve_j \ve_j^\top \end{bmatrix}, \Gtilde_j = \begin{bmatrix} -U_j^2 &  \vo^\top \\ \vo & \ve_j \ve_j^\top \end{bmatrix}, \bm{H} = \begin{bmatrix} 1 & \vo^\top \\ \vo & \vo \end{bmatrix}.
    \]
    
    Note that $\ve_j$ is a vector with 1 at the $j$th entry and 0's elsewhere.

    Following similar arguments as in Steps 1 and 2, the feasible region of $\mc{P}^\hom$ is nonempty and bounded. In addition, $\mc{P}^\hom$ has a hyperplane constraint \eqref{eq: hom_sdp_4} and two conic constraints \eqref{eq: hom_sdp_5} and $\Ytilde\in\kcone$, where $\kcone$ is the homogenized constraint cone defined by $\kcone := \{\Ytilde | \Ytilde~\text{satisfies}~\eqref{eq: hom_sdp_1}-\eqref{eq: hom_sdp_3}, \eqref{eq: hom_sdp_6}\}$. Therefore, by Corollary 1.2 of \citep{kim2025strong}, $\mc{P}^\hom$ has strong duality.\\

    \noindent{\it Step 4: }Unlike for CPPs where equivalence holds between the homogenized relaxation and the original formulation, for our SDP relaxation this is not true in general. Instead, we prove that $\opt(\mc{P}^\hom)\geq \opt(\mc{P}^\sdpbd)$. It suffices to show that any feasible solution $\Ytilde$ of $\mc{P}^\hom$ is also feasible to $\mc{P}^\sdpbd$. In particular, we show that $\Ytilde$ satisfies \eqref{eq: sdp_2} and \eqref{eq: sdp_3}, as other constraints are identical. 

    Denote $\Ytilde = \begin{bmatrix}1 & \vxtilde^\top \\ \vxtilde & \Xtilde \end{bmatrix}$. Since $\Ytilde\succeq 0$, its Schur complement $\Xtilde - \vxtilde \vxtilde^\top \succeq 0$. Thus, $\va_j^\top \Xtilde \va_j \ge (\va_j^\top \vxtilde)^2$. By \eqref{eq: hom_sdp_1}, $0 = \va_j^\top \Xtilde\va_j - 2b_j\va_j^\top \vxtilde + b_j^2$. Therefore,
    \[
    0 = \va_j^\top \Xtilde\va_j - 2b_j\va_j^\top \vxtilde + b_j^2 \ge (\va_j^\top \vxtilde)^2 - 2b_j\va_j^\top \vxtilde + b_j^2 = (\va_j^\top \vxtilde - b_j)^2 \Rightarrow \va_j^\top \vxtilde - b_j = 0.
    \]
    Thus, $\Ytilde$ satisfies \eqref{eq: sdp_2}. Substituting $\va_j^\top \vxtilde = b_j$ back to $0 = \va_j^\top \Xtilde\va_j - 2b_j\va_j^\top \vxtilde + b_j^2$, we have $\va_j^\top \Xtilde \va_j - b_j^2 = 0$, proving that $\Ytilde$ satisfies \eqref{eq: sdp_3}. 
    
    Therefore, $\Ytilde$ is feasible to $\mc{P}^\sdpbd$ and thus $\opt(\mc{P}^\hom)\geq \opt(\mc{P}^\sdpbd)$.\\

    \noindent{\it Step 5: }Finally, we prove that the optimal value of $\mc{P}^{\sdpdual}$ is no less than the optimal value of the dual for $\mc{P}^\hom$ (denoted as $\mc{P}^{\homdual}$), and thus strong duality between $\mc{P}^\hom$ and its dual ensures the same for $\mc{P}^{\sdpdual}$. The argument follows the proof of Theorem 6.2 in \cite{brown2024copositive} for CPPs, with modifications to account for the different dual formulation due to different primal constraints in SDPs. 
    
    The formulation of $\mc{P}^{\sdpdual}$ is as follows:
    \begin{subequations}\label{eq: sdpdual_full}
    \begin{align}
    \mc{P}^{\sdpdual}: \max~~& \sum_{j=1}^m b_j \gamma_j + \sum_{j=1}^m b^2_j \omega_j + \mu + \sum_{j\in [n]\sm \mc{B}} U_j^2 \phi_j\\
    \st ~~& M(\pmb \gamma, \pmb \omega, \pmb \kappa, \pmb \phi, \mu) = \bm{\Theta} + \bm{\Omega}\\
    & \pmb\phi\le \vo, \bm{\Theta} \geq \vo, \bm{\Omega}\in\mc{S}^{n+1}_+,
    \end{align}
    \end{subequations}
    where $M(\pmb \gamma, \pmb \omega, \pmb \kappa, \pmb \phi, \mu) = \bm{C} - \sum_{j=1}^m \bm{A}_j \gamma_j - \sum_{j=1}^m \hat{\bm{A}}_j \omega_j - \sum_{j\in\mc{B}} \bm{B}_j\kappa_j - \sum_{j\in [n]\sm\mc{B}} \bm{G}_j \phi_j - \bm{H}\mu$, with 
    \[
    \bm{A}_j = \begin{bmatrix} 0 & \frac{1}{2}\va_j^\top \\ \frac{1}{2}\va_j & \vo \end{bmatrix}, \hat{\bm{A}}_j = \begin{bmatrix} 0 & \vo^\top \\ \vo & \va_j \va_j^\top \end{bmatrix}, \bm{G}_j = \begin{bmatrix} 0 &  \vo^\top \\ \vo & \ve_j \ve_j^\top \end{bmatrix}.
    \]
    
    Also, $\mu$ denotes the dual variable of $\tr(\bm{H} \Ytilde) = 1$, which is a valid constraint for $\mc{P}^\sdpbd$.

    The formulation of $\mc{P}^{\homdual}$ is as follows:
    \begin{subequations}\label{eq: hom_dual}
    \begin{align}
    \mc{P}^{\homdual}: \max~~& \mutilde\\
    \st ~~& \tilde{M}(\pmb \sigmatilde, \pmb \kappatilde, \pmb \phitilde, \mutilde) = \Thetatilde + \Omegatilde\\
    & \pmb\phitilde\le \vo, \Thetatilde \geq \vo, \Omegatilde\in\mc{S}^{n+1}_+,
    \end{align}
    \end{subequations}
    where $\tilde{M}(\pmb \sigmatilde, \pmb \kappatilde, \pmb \phitilde, \mutilde) = \bm{C} - \sum_{j=1}^m \Atilde_j \sigmatilde_j - \sum_{j\in\mc{B}} \Btilde_j\kappatilde_j - \sum_{j\in [n]\sm\mc{B}} \Gtilde_j \phitilde_j - \bm{H}\mutilde$.

    To prove that $\opt(\mc{P}^{\sdpdual})\ge \opt(\mc{P}^{\homdual})$, it suffices to show that any feasible solution $(\pmb \sigmatilde, \pmb \kappatilde, \pmb \phitilde, \mutilde, \Thetatilde, \Omegatilde)$ of $\mc{P}^{\homdual}$ can be mapped to a feasible solution of $\mc{P}^{\sdpdual}$ with the same objective value. In $\mc{P}^{\sdpdual}$, let $\gamma_j = -2b_j\sigmatilde_j, \omega_j = \sigmatilde_j, \kappa_j = \kappatilde_j, \phi_j = \phitilde_j, \Theta = \Thetatilde, \Omega = \Omegatilde$, and $\mu = \mutilde + \sum_{j=1}^m b_j^2 \sigmatilde_j - \sum_{j\in [n]\sm\mc{B}} U^2_j \phitilde_j$. Then we can verify that this solution leads to an objective value of $\mutilde$ and is feasible to $\mc{P}^{\sdpdual}$. For feasibility, it is useful to note that $\Atilde_j = \hat{\bm{A}}_j - 2b_j \bm{A}_j + b_j^2 \bm{H}$ and $\Gtilde_j = \bm{G}_j - U_j^2 \bm{H}$. 
    
    Therefore, $\opt(\mc{P}^{\sdpdual})\ge \opt(\mc{P}^{\homdual})$. Combined with $\opt(\mc{P}^\hom) = \opt(\mc{P}^{\homdual})$ (Step 3), $\opt(\mc{P}^\hom)\geq \opt(\mc{P}^\sdpbd)$ (Step 4), and $\opt(\mc{P}^{\sdpdual}) \le \opt(\mc{P}^{\sdpbd})$ (weak duality), we conclude that strong duality holds for $\mc{P}^\sdpbd$.
    }
    \qed
    \endproof
\begin{proof}{Proof of Theorem \ref{prop:envelopetheorem1}}
    {Fix $\theta \in [\underline{\theta}, \overline{\theta}] $ and let $\Psi^*\in \mathrm{Feas}^*(P^{\mathrm{SDP\text{-}dual}}(\theta))$ be any dual optimal solution. Then, since strong duality is attained (Proposition \ref{th: dual}), we have
\[
V^{\mathrm{SDP}}(\theta) \;=\; V^{\mathrm{SDP\text{-}dual}}(\theta) \;=\; f^d(\Psi^*,\theta),
\qquad
V^{\mathrm{SDP}}(\theta') \;=\; \max_{\Psi\in \feas(\psdpdual(\theta))} f^d(\Psi,\theta')
\]
for every $\theta' \in [\underline{\theta}, \overline{\theta}]$. We emphasize that since $\theta$ appears in the right-hand sides of $\psdpbd(\theta)$, the dual feasible set, $\feas(\psdpdual(\theta))$, does not depend on $\theta$. Therefore, \citet[Theorem 1]{milgrom2002envelope} applies to $V^{\mathrm{SDP\text{-}dual}}(\theta)$ and yields the envelope inequalities for one-sided derivatives. In particular, we have:

\paragraph{Right derivative.}
Let $\theta'>\theta$.
Since $\Psi^*$ is feasible for the dual at $\theta'$, we have
\[
V^{\mathrm{SDP}}(\theta') \;=\; \max_{\Psi\in \feas(\psdpdual(\theta))} f^d(\Psi,\theta') \;\ge\; f^d(\Psi^*,\theta').
\]
Subtracting $V^{\mathrm{SDP}}(\theta)=f^d(\Psi^*,\theta)$ and dividing by $\theta'-\theta>0$ yields
\[
\frac{V^{\mathrm{SDP}}(\theta')-V^{\mathrm{SDP}}(\theta)}{\theta'-\theta}
\;\ge\;
\frac{f^d(\Psi^*,\theta')-f^d(\Psi^*,\theta)}{\theta'-\theta}.
\]
Since $b_j(\theta), j=1,...,m$ are differentiable and $f^d(\Psi^*,\theta)$ is a polynomial in $b_j$'s, $f_{\theta}^d(\Psi^*,\theta)$ exists.  Taking $\theta'\rightarrow\theta+$ at $\theta$ gives
\[
V^{\mathrm{SDP}'}(\theta+) \;\ge\; f_{\theta}^d(\Psi^*,\theta).
\]

\paragraph{Left derivative.}
Let $\theta'<\theta$.
Again, $V^{\mathrm{SDP}}(\theta')\ge f^d(\Psi^*,\theta')$ and $V^{\mathrm{SDP}}(\theta)= f^d(\Psi^*,\theta)$ imply
\[
V^{\mathrm{SDP}}(\theta)-V^{\mathrm{SDP}}(\theta')
\;\le\;
f^d(\Psi^*,\theta)-f^d(\Psi^*,\theta').
\]
Dividing by $\theta-\theta'>0$ and taking $\theta'\rightarrow\theta-$ yields
\[
V^{\mathrm{SDP}'}(\theta-) \;\le\; f_{\theta}^d(\Psi^*,\theta).
\]

\paragraph{Differentiable points.}
If $V^{\mathrm{SDP}}$ is differentiable at $\theta$, then $V^{\mathrm{SDP}'}(\theta-)=V^{\mathrm{SDP}'}(\theta+)=V^{\mathrm{SDP}'}(\theta)$.
Combining the two inequalities above gives
\[
V^{\mathrm{SDP}'}(\theta)=f_{\theta}^d(\Psi^*,\theta). \quad \qed
\]}

   \end{proof}

\begin{proof}{Proof of Corollary \ref{cor:vector-envelope}}
First, due to strong duality of $\psdpbd$, we have $V(\vu)=\\ \max_{\Psi\in\mathrm{Feas}(P^{\mathrm{SDP\text{-}dual}}(\vu))} f^d(\Psi,\vu)$. Since $f^d(\Psi^*,\vu)$ is a polynomial in $b_j, j=1,...,m$, and $b_j(\vu)'s$ are differentiable in $\vu$, $\nabla_{\vu} f^d(\Psi^*,\vu)$ exists. 

To obtain the expression for $\nabla V(\vu)$, we first consider a perturbation along one direction. Fix $\vu=\vuhat\in\mathcal U$. Let $\vh\in\mathbb R^d$ be any direction and let $\underline{\varepsilon}, \overline{\varepsilon} >0$ be such that
$\vuhat+\varepsilon \vh\in\mathcal U$ for all $\varepsilon\in[-\underline{\varepsilon},\overline{\varepsilon}]$.
Define the one-dimensional value function along $\vh$ by
$\phi(\varepsilon):=V(\vuhat+\varepsilon \vh)$. If $V(\vuhat)$ is differentiable at $\vuhat$, then $\phi(\varepsilon)$ is differentiable at $\varepsilon = 0$. 
Also, if $V(\vuhat)$ is differentiable at $\vuhat$, then for every direction $\vh$ we have 
\begin{align}\label{eq: phi_0}
\phi'(0)=\nabla V(\vuhat)^\top \vh.
\end{align}



Define $\vu(\varepsilon):=\vuhat+\varepsilon \vh$. The dual objective function can be written as $f^d(\Psi,\vu(\varepsilon))$. By Theorem \ref{prop:envelopetheorem1} and the chain rule, if $V(\vuhat)$ is differentiable at $\vuhat$, then 
\begin{align}\label{eq: phi}
\phi'(0) = f^d_{\varepsilon}(\Psi^*,\vu(\varepsilon))\Big|_{\varepsilon=0}
=\nabla_{\vu} f^d(\Psi^*,\vu)^\top \vh.
\end{align}

Combining \eqref{eq: phi_0} and \eqref{eq: phi} for all $\vh$ yields $\nabla V(\vuhat)=\nabla_{\vu} f^d(\Psi^*,\vuhat)$.
\qed
\end{proof}
    \proof{Proof of Proposition \ref{th: foc_d}}
    The price $r^d_{kt}$ is obtained by taking the partial derivative of $f^{\rm uc-d}(\Psi^*, D_{kt})$ with respect to $D_{kt}$. Since we will take the derivative with respect to $D_{kt}$ and the constant terms do not impact the derivative, we only write the terms of $f^{\rm uc-d}(\Psi^*, D_{kt})$ that include $D_{kt}$ for ease of exposition: 

    \begin{align} \label{dkt_deriv_fixed_k}
        &\lambda^*_t D_{kt} + \hat{\lambda}^*_t D_{kt}^2 + 2\hat{\lambda}^*_t\sum\limits_{\substack{k' \in \mc{N}\\ k' \neq k }}D_{kt}D_{k't}  + \sum\limits_{(i,j)\in \mc{L}}B_{ij}(S_{ik} - S_{jk}) (\xi_{ijt}^{\min*} + \xi_{ijt}^{\max*})D_{kt} \nonumber\\
        & + \sum\limits_{(i,j) \in \mc{L}} B_{ij}^2 (S_{ik} - S_{jk})^2 D_{kt}^2(\hat{\xi}_{ijt}^{\min*} + \hat{\xi}_{ijt}^{\max*}) \nonumber\\
        & + 2\sum\limits_{(i,j) \in \mc{L}} B_{ij}^2 \left(\sum\limits_{\substack{k' \in \mc{N}\\ k' \neq k }}(S_{ik} - S_{jk})(S_{ik'} - S_{jk'})D_{k't}D_{kt}\right) (\hat{\xi}_{ijt}^{\min*} + \hat{\xi}_{ijt}^{\max*}) \nonumber\\
        & +2\sum\limits_{(i,j) \in \mc{L}} B_{ij}(S_{ik} - S_{jk})(P^{\text{Trans},\min}_{ij}\hat{\xi}^{\min*}_{ijt} + P^{\text{Trans},\max}_{ij}\hat{\xi}^{\max*}_{ijt})D_{kt},  
    \end{align}
    where the terms are obtained by multiplying the right-hand side of relevant parameter in \eqref{eq: uc_sdp} with corresponding dual variables. Specifically, the first term is from \eqref{eq: uc_sdp_1}, the second and third terms are from \eqref{eq: uc_sdp_5}, the fourth term is from \eqref{eq: uc_sdp_2} and \eqref{eq: uc_sdp_3}, and the last three terms are from \eqref{eq: uc_sdp_6} and \eqref{eq: uc_sdp_7}.  
    
    Taking the derivative of \eqref{dkt_deriv_fixed_k} with respect to $D_{kt}$, we get \eqref{eq: foc_d}.  \qed\endproof
\begin{proof}{Proof of Corollary \ref{cor:multi-demand-abscont}}
The proof follows similar steps as the proof of Proposition \ref{th: continuous}, with $\theta$ as the scalar parameter. Namely, let $\bm{w}$ collect all primal decision variables of $\uc^\sdpbd((D(\theta))$ (including lifted matrix
entries), and rewrite the objective as $\vc^\top \vw$.
Then, the analogue of the set $S$ in the proof of Proposition \ref{th: continuous},
\[
S:=\{(\theta,w,z)\in \R\times\mathbb{R}^p\times\mathbb{R}:\ \theta\in [\underline\theta,\overline\theta], w\in\mathrm{Feas}(D(\theta)),\ z=\vc^\top \vw\},
\]
is semialgebraic, and so is its projection onto $(\theta,z)$.
Analogously to the proof of Proposition \ref{th: continuous}, this implies $V^{\mathrm{UC\text{-}SDP}}(D(\theta))\bigr)$
is a semialgebraic function. The rest of the proof follows the same argument as {\cgreen Steps 2} in the proof of Proposition \ref{th: continuous}. 
\qed
\end{proof}
\section{Proofs for Section \ref{ch: ac}}\label{ch: proof_2}
\proof{Proof of Proposition \ref{th: acuc_strong_dual}}
    Our proof follows a similar argument to the proof of Proposition \ref{th: dual}.

    First, similarly to the argument in Proposition \ref{th: dual}, if the original ACUC problem $\ac$ is feasible, then so is $\ac^\sdp$, via a rank-one lifting argument. 
    
    Second, since the diagonal entries of $\bm{W}_t$'s are nonnegative and bounded by \eqref{eq: rectangular_5}, and any off-diagonal entries $W_{ijt}^2\leq W_{iit}W_{jjt}$ (due to nonnegativity of $2\times 2$ principal minors for $W_t\succeq 0$), entries of $\bm{W}_t$'s are bounded. In addition, other variables are also bounded following the same argument in Proposition \ref{th: dual}. 
    The feasible region is also closed, since it is the intersection of finitely many closed sets: {\cgreen closed convex cones and a set with finitely many linear equalities and inequalities}. Combined with the boundedness result shown above, the feasible region is thus compact.
        
    {\cgreen Following similar arguments as Steps 3 of Proposition \ref{th: dual}'s proof, we can formulate a homogenized SDP relaxation with a block-diagonal matrix that includes all variables. The feasible region is nonempty, closed, and consists of one hyperplane (i.e., $\Ytilde_{11} = 1$) and two cones. Specifically, one of the cones is defined by all conic constraints, and the other is defined by all homogenized linear constraints. Thus, strong duality still follows for this homogenized SDP relaxation. 

    The optimal objective of the homogenized problem is still no less that that of the original SDP relaxation. The feasible region of the homogenized problem is contained in that of the original problem, as the homogenized constraints reduce to the original ones.

    Finally, any feasible dual solution of the homogenized problem can be mapped to a feasible dual solution of $\ac^\sdp$ with the same objective value by (1) keeping the dual variables associated with shared constraints unchanged, and (2) absorbing the constant terms introduced by homogenization into the multiplier of the hyperplane constraint $\Ytilde_{11} = 1$. Therefore, the optimal solution of $\ac^\sdp$'s dual is no less than that of the homogenized problem's dual. Combining this with previous results and with weak duality for $\ac^\sdp$ yields strong duality for $\ac^\sdp$ and its dual.}
    \qed
    \endproof
    \proof{Proof of Proposition \ref{th: acuc_price}}{
    Observe that the parameter $D_{kt}$ appears only on the right-hand sides of \eqref{eq: ac_sdp_2} and \eqref{eq: ac_sdp_4}. Hence, in the dual objective, the dependence on $D_{kt}$ appears entirely  through the terms
    \[
    \lambda^{R*}_{kt}D^R_{kt}+\lambda^{I*}_{k't}D^I_{k't}+
    \hat\lambda^{R*}_t\Bigl(\sum_{k'\in N}D^R_{k't}\Bigr)^2,
    \]
    up to an additive constant independent of $D_{kt}$. Thus, differentiating with respect to $D^R_{kt}$ gives
    $r^{d\text{-}R}_{kt}=\lambda^{R*}_{kt}+2\hat\lambda^{R*}_t\sum_{k'\in N}D^R_{k't}$,
    and differentiating with respect to $D^I_{kt}$ gives $r^{d\text{-}I}_{kt}=\lambda^{I*}_{kt}$. 
    \qed}
    \endproof
\section{Proofs for Section \ref{ch: loc}}\label{ch: proof_3}
\proof{Proof of Lemma \ref{th: conjugate}}
    We have:
    \begin{align*}
        V^c(\vpi) := &\sup_{\vD\in\R^{|\mc{I}|}} \left\{\vpi^\top \vD - V(\vD)\right\} \\
         = & \sup_{\vD\in\R^{|\mc{I}|}} \left\{\vpi^\top \vD + \max_{\substack{\vx_g\in \X_g, \forall g\in\G;\\ \vr\in\mc{R};\\ \bm{A}^s \vs - \vr = \vD}} - \sum_{g\in\G} C_g(\vx_g)\right\}\\
         = & \sup_{\vD\in\R^{|\mc{I}|}} \max_{\substack{\vx_g\in \X_g, \forall g\in\G;\\ \vr\in\mc{R};\\ \bm{A}^s \vs - \vr = \vD}}\left\{\vpi^\top \vD - \sum_{g\in\G} C_g(\vx_g)\right\}\\
         = & \sup_{\vD\in\R^{|\mc{I}|}} \max_{\substack{\vx_g\in \X_g, \forall g\in\G;\\ \vr\in\mc{R};\\ \bm{A}^s \vs - \vr = \vD}}\left\{\vpi^\top (\bm{A}^s \vs - \vr) - \sum_{g\in\G} C_g(\vx_g)\right\}\\
         = & \sup_{\substack{\vx_g\in \X_g, \forall g\in\G;\\ \vr\in\mc{R}}} \left\{\vpi^\top (\bm{A}^s \vs - \vr) - \sum_{g\in\G} C_g(\vx_g)\right\}\\
         = & \sup_{\vx_g\in \X_g, \forall g\in\G}\left(\vpi^\top \bm{A}^s \vs - \sum_{g\in\G} C_g(\vx_g)\right)+\sup_{\vr\in\mc{R}} \{-\vpi^\top \vr\}\\
         = & \sum_{g\in\G} \max_{\vx_g\in\X_g}\left(\pi_{k(g)}^\top \vs_g - C_g(\vx_g)\right) + \sup_{\vr\in\mc{R}} \{-\vpi^\top \vr\},
    \end{align*}
    where the second equality follows from the definition of $V(\vD)$. The third equality holds because $\vpi^\top \vD$ is constant with respect to the inner maximization problem. The fourth equality is due to the constraint $\bm{A}^s \vs - \vr = \vD$. 

    The fifth equality holds because the outer supremum is taken over \(\vD \in \mathbb{R}^{|\mc{I}|}\), and \(\vD = \bm{A}^s \vs - \vr\). Thus, taking the supremum over all feasible \(\vx_g \in \X_g\) for all \(g \in \G\) and \(\vr \in \mc{R}\) implicitly spans all values of \(\vD\) that can possibly lead to a supremum, i.e., those $\vD$ corresponding to feasible solutions. Any \(\vD\) outside this range would correspond to an infeasible inner maximization problem with a value of $-\infty$.

    The sixth equality is due to the feasible regions \(\vx_g \in \X_g, \forall g\in\G\) and $\vr\in\mc{R}$ being separable. The final equality holds because, for each bus $k\in\N$ and time $t\in\T$, the term $[\bm{A}^s \vs]_{(kt)}$ represents the total generation at bus $k$, obtained by summing the production from all generators located at that bus. Thus, $\vpi^\top \bm{A}^s \vs = \sum_{g\in\G} \pi_{k(g)}^\top \vs_g$, which corresponds to the total revenue from the production of all generators. This argument remains valid in the AC setting, where $\bm{A}^s \vs$ includes both real and reactive power generation.
    \qed
    \endproof
    \proof{Proof of Corollary \ref{th: loc_eq}}
    The bound is tight if the two inequalities of \eqref{eq: loc_proof} in the proof of Theorem \ref{th: loc} become equalities. The first inequality is due to the order-reversing property of the conjugate, and thus becomes equality when $V^c(\vpi)=(\vrel)^c(\vpi)$. 
    
    The second inequality follows from $\sup_{\vr\in\mc{R}}\{-\vpi^\top \vr\} - (-\vpi^\top \vrfeas)\geq 0$. When there is no congestion, the flow constraints in $\mc{R}$ become redundant, allowing \eqref{eq: min_cost} to be reformulated without $\vr$ terms and constraints \eqref{eq: min_cost} to be aggregated across all buses. As a result, the $\vr$ and $\vrfeas$ terms are dropped from \eqref{eq: loc_proof}, eliminating the second inequality.\qed
    \endproof
\section{Triangle Inequalities}\label{ch: ineq}
In our experiments, the following triangle inequalities are imposed on binary variables $z_{gt}$ ($\forall g_1,g_2,g_3 \in \mc{G}$ and $t\in \mc{T}$): 
 \begin{align}
     &z_{g_1t}z_{g_2t} + z_{g_3t}\geq z_{g_1t}z_{g_3t} + z_{g_2t}z_{g_3t} \label{triangle1} \\ 
     &z_{g_1t}z_{g_2t} + z_{g_1t}z_{g_3t} + z_{g_2t}z_{g_3t} + 1 \geq z_{g_1t} + z_{g_2t} + z_{g_3t}. \label{triangle2}
 \end{align}
 
 Since those inequalities contain bilinear terms, we linearize them by replacing the bilinear terms with corresponding lifted terms.

%

\end{APPENDIX}
%
%






%
%
%

\end{document}